\documentclass[12pt]{tac}

\usepackage[T1]{fontenc}
\usepackage[utf8]{inputenc}
\usepackage{amsmath}
\usepackage{amssymb}
\usepackage[all]{xy}
\usepackage{amscd,verbatim}
\usepackage{enumitem}

\usepackage{color}

\newcommand{\Hom}{\operatorname{Hom}}

\newcommand{\St}{\operatorname{St}}
\newcommand{\End}{\operatorname{End}}

\newcommand{\Spec}{\operatorname{Spec}}

\newcommand{\Gal}{\operatorname{Gal}}
\newcommand{\Id}{\operatorname{Id}}
\newcommand{\res}{\operatorname{res}}
\newcommand{\cores}{\operatorname{cores}}
\renewcommand{\lim}{\varprojlim}
\newcommand{\colim}{\varinjlim}

\newcommand{\sA}{\mathcal{A}}
\newcommand{\sB}{\mathcal{B}}
\newcommand{\sC}{\mathcal{C}}
\newcommand{\sD}{\mathcal{D}}
\newcommand{\sE}{\mathcal{E}}

\newcommand{\sM}{\mathcal{M}}

\newcommand{\C}{\mathbf{C}}

\newcommand{\G}{\mathbb{G}}

\renewcommand{\P}{\mathbf{P}}
\newcommand{\Q}{\mathbf{Q}}

\newcommand{\Z}{\mathbf{Z}}

\newcommand{\usA}{\mathbf{\mathcal{A}}}
\newcommand{\usB}{\mathbf{\mathcal{B}}}
\newcommand{\un}{\mathbf{1}}

\newcommand{\Ex}{{\operatorname{\mathbf{Ex}}}}
\newcommand{\Funct}{\operatorname{\mathbf{Funct}}}
\renewcommand{\Vec}{\operatorname{\mathbf{Vec}}}
\newcommand{\Lact}{\operatorname{\mathbf{Lact}}}
\newcommand{\Rep}{\operatorname{\mathbf{Rep}}}
\newcommand{\LCorr}{\operatorname{\mathbf{LCorr}}}
\newcommand{\LMot}{\operatorname{\mathbf{LMot}}}
\newcommand{\Mot}{\operatorname{\mathbf{Mot}}}
\newcommand{\MM}{\operatorname{\mathbf{NMot}}}

\newcommand{\Aut}{\operatorname{\mathbf{Aut}}}
\newcommand{\Cat}{\operatorname{\mathbf{Cat}}}
\newcommand{\Sm}{\operatorname{\mathbf{Sm}}}

\newcommand{\Res}{\operatorname{Res}}
\newcommand{\Ind}{\operatorname{Ind}}
\newcommand{\Ker}{\operatorname{Ker}}
\newcommand{\Coker}{\operatorname{Coker}}
\newcommand{\IM}{\operatorname{Im}}
\newcommand{\Nori}{{\operatorname{Nori}}}
\newcommand{\op}{{\operatorname{op}}}
\newcommand{\proj}{{\operatorname{proj}}}

\newcommand{\gal}{{\operatorname{gal}}}

\renewcommand{\epsilon}{\varepsilon}
\renewcommand{\phi}{\varphi}

\newcommand{\et}{{\operatorname{\acute{e}t}}}
\newcommand{\tr}{\operatorname{tr}}
\newcommand{\rig}{{\operatorname{rig}}}

\newcommand{\inj}{\hookrightarrow}

\newcommand{\iso}{\overset{\sim}{\longrightarrow}}
\newcommand{\by}{\xrightarrow}

\newcommand{\qed}{\hfill \small $\blacksquare$}

\newtheorem{Thm}{Theorem}
\newtheorem{thm}{Theorem}[section]
\newtheorem{prop}[thm]{Proposition}
\newtheorem{cor}[thm]{Corollary}
\newtheorem{lemma}[thm]{Lemma}
\theoremstyle{definition}
\newtheorem{defn}[thm]{Definition}
\theoremstyle{remark}
\newtheorem{rk}[thm]{Remark}
\newtheorem{rks}[thm]{Remarks}

\newtheorem{ex}[thm]{Example}

\newcounter{spec}
\newenvironment{thlist}{\begin{list}{\rm{(\roman{spec})}}%
{\usecounter{spec}\labelwidth=20pt\itemindent=0pt\labelsep=10pt}}%
{\end{list}}%
\numberwithin{equation}{section}
\begin{document}
\title{Galois descent for motivic theories}
\author{Bruno Kahn}
\address{CNRS, Sorbonne Université and Université Paris Cité, IMJ-PRG\\ Case 247\\4 place
Jussieu\\75252 Paris Cedex 05\\France}
\eaddress{bruno.kahn@imj-prg.fr}
\date{May 30, 2024}
\keywords{Stacks, motives, Tannakian categories}
\amsclass{18F20, 18M25, 14C15}
\maketitle
\begin{abstract} We give necessary conditions for a category fibred in pseudo-abelian additive categories over the classifying topos of a profinite group to be a stack; these conditions are sufficient when the coefficients are $\Q$-linear. This applies to pure motives over a field in the sense of Grothendieck, Deligne-Milne and André, to mixed motives in the sense of Nori and to several motivic categories considered in \cite{adjoints}. We also give a simple proof of the exactness of a sequence of motivic Galois groups under a Galois extension of the base field, which applies to all the above (Tannakian) situations. Finally, we clarify the construction of the categories of Chow-Lefschetz motives given in \cite{chow-lefschetz} and simplify the computation of their motivic Galois group in the numerical case.  
\end{abstract}

\begin{flushright}{\it 
Du blanc! Verse tout, verse de par le diable! Verse deça tout plein, la langue me pelle.

-- Lans. tringue.

-- A toy, compaing! De hayt! de hayt! 

-- La! là! là! C'est morfiaillé, cela.

-- O lachryma Christi!

-- C'est de la Deviniere, c'est vin
pineau! 

-- O le gentil vin blanc! 

-- Et par mon ame, ce n'est
que vin de tafetas. 

-- Hen, hen, il est à une aureille, bien drappé
et de bonne laine.}
\medskip

Rabelais, Gargantua, ch. V.
\end{flushright}

\section{Introduction} The first main result of this article is

\begin{Thm}\label{t0} Let $k$ be a field, $\sim$ an adequate equivalence relation on algebraic cycles with rational coefficients and $\Mot_\sim(k)$ the category of pure motives over $k$ modulo $\sim$, in the sense of Grothendieck. Then the assignment
\[l\mapsto \Mot_\sim(l)\]
defines a stack of rigid $\otimes$-categories over the small étale site of $\Spec k$.
\end{Thm}

Theorem \ref{t0} is so easy to prove that it ought to be part of the folklore. Here is a sketch: for $l/k$ a finite Galois extension with group $\Gamma$, write $\Mot_\sim(l)[\Gamma]$ for the category of descent data on $\Mot_\sim(l)$ relative to $\Gamma$. We have to prove that the canonical functor $\Mot_\sim(k)\to \Mot_\sim(l)[\Gamma]$ is an equivalence of categories. Full faithfulness follows from a standard transfer argument, using the facts that the base change functor $f^*:\Mot_\sim(k)\to \Mot_\sim(l)$ has a right adjoint $f_*$ and that the coefficients are $\Q$. For the essential surjectivity,  if $(C,(b_g)_{g\in \Gamma})\in  \Mot_\sim(l)[\Gamma]$ is a descent datum, with $C\in \Mot_\sim(l)$, the natural action of $\Gamma$ on $f_*C$ gives a projector whose image yields the effectivity of the descent datum.

In fact, such a result and sketch of proof hold in much greater generality, which led me to give them an abstract formulation: this is the purpose of Section \ref{s1}. In  Theorem \ref{p5.3}, we get necessary  conditions for a fibered category in additive categories over the classifying topos of a profinite group to be a stack; they are sufficient when the categories are pseudo-abelian and $\Q$-linear. These conditions use a baby ``2 functor formalism'' (for Galois étale coverings!), see Definition \ref{d5.2}. In Corollary \ref{c2.1}, we show how to weaken some hypotheses of this formalism in the presence of a monoidal structure: this allows us to easily prove the stack property for all motivic theories appearing in \cite[Th. 4.3 a)]{adjoints}, not just for pure motives (Theorem \ref{t3.1}). It also applies to the related theories of \cite{dm} and \cite{pour} and to Nori motives \cite{h-ms}. 

What started me on this work was the desire to clarify and simplify a construction and a reasoning in \cite{chow-lefschetz}, and descent arguments in its sequel \cite{ff}. In  \cite[§4]{chow-lefschetz}, I construct a category of Chow-Lefschetz motives over a (possibly non separably closed) field in two steps: first a ``crude'' category and then a better-behaved one. By hindsight, it became likely that the second step was just the process of creating the associated stack, and this is what is checked in Proposition \ref{p4.1}. The reasoning I wanted to simplify was the rather ugly recourse to continuous descent data in the proof of \cite[Th. 5]{chow-lefschetz}: this is done here in Theorem \ref{t5.1}, which also clarifies the proof of \cite[Prop. 6.23 (a)]{dm}\footnote{What is not clear in this proof is why the map of motives $\Hom(\bar M,\bar N)\inj  \underline{\Hom}(M,N)$ on p. 215 exists. For simplicity, take $M=\un$, so that the purported inclusion reads $\Hom(\un,\bar N)\inj N$.  If we think in terms of representations of the motivic Galois groups, the left hand side is the invariants of the right hand side under the action of the geometric Galois group $G^0(\sigma)$. It is a subrepresentation of $N$ provided $G^0(\sigma)$ is \emph{normal} in the arithmetic Galois group $G(\sigma)$.} (quoted without comment in \cite[4.6, exemples]{pour}). A semisimplicity assumption which appeared in the first version of this paper has now been dropped from this theorem, which makes it also applicable to \cite[Th. 4.7]{jannsen} and \cite[Th. 9.1.16]{h-ms}. Concerning \cite{ff}, the reasonings of \S\S 4 and 5 in its first version are greatly clarified by \S \ref{s4} of the present paper.

Note that Theorem \ref{t0} does not extend to motives over a base $S$ in the sense, say, of Deninger-Murre \cite{den-mur}; indeed, for $X,Y$ smooth projective over $S$, the presheaf $U\mapsto CH^*(X_U\times_U Y_U)_\Q$ for $U\to S$ étale is already not a sheaf in the Zariski topology! Similarly, Theorem \ref{t0} is obviously false if the coefficients are not $\Q$-linear (think of the Néron-Severi group of an anisotropic conic). If one wanted to extend it to these two cases, one would probably have to consider (stable?) $\infty$-categories.  Hopefully, the results of this paper will give an insight on what to do in that situation. Indeed, one may wonder if a suitable subset of a six functors formalism can be used to imply descent in the present spirit.

\subsection*{Structure of the paper} It is divided in two parts, plus an appendix. Part I contains foundational material: it deals with Galois descent theory from the most general (additive categories) to the most particular (Tannakian categories). Part II concerns applications to various motivic theories: some of those considered in \cite{adjoints}, Lefschetz motives à la Milne from \cite{chow-lefschetz}, $1$-motives and Nori motives \cite{h-ms}.

Section \ref{s1} concerns additive categories without extra structure; actually, the additivity hypothesis only appears from Subsection \ref{s1.4} onwards. The main result, Theorem \ref{p5.3}, says that descent is basically equivalent to a two functors formalism: right adjoints, a base change isomorphism and a trace structure (see Proposition \ref{p5.3a} for a more precise statement). Two other important ingredients are the construction of a retraction (Lemma \ref{l1.1} b)) and a monadic approach (\S\S \ref{s1.3} and \ref{s1.2}): both play a key rôle later. Section \ref{s.mon} adds $\otimes$-structures to the situation; this allows us to simplify the two functors axioms, yielding some conditions which are easy to verify in practice (Corollary \ref{c2.1}).

In Section \ref{s3}, we show that a diagram of $\otimes$-categories
\begin{equation}\label{eq1.3}\tag{*}
\begin{CD}
\sA'\\
@A f^*AA\\
\sA@>\gamma>> \sB,
\end{CD}
\end{equation}
where $f^*$ is a descent functor in the sense outlined above, has a categorical push-out in a fashion: see Propositions \ref{p4.2} and \ref{p2.3}. This uses the monadic approach alluded to above.

Section \ref{s4} does two things. Suppose given a diagram \eqref{eq1.3} of Tannakian categories over a field $K$, with $\gamma$ faithful and exact, whence a push-out $\sB'$ as above. Consider the fibre functor $\omega_\sB:\sB\to \Vec_L$, where $L$ is an extension of $K$. Replacing $\gamma$ by $\omega=\omega_\sB\circ \gamma$ in  \eqref{eq1.3}, we get a new diagram  whose push-out turns out to be of the form $\Vec_R$ for an étale $L$-algebra $R$, and the universal property provides a ``fibre functor'' $\omega'_\sB:\sB'\to \Vec_R$ and an $L$-homomorphism $R\to L$ if $\omega$ extends to $\sA'$. The first main result, Theorem \ref{t5.1}, is that when $\sB=\sA$ and $L=K$ the corresponding sequence of Tannakian groups is exact. The second main result, Proposition \ref{p4.3}, gives a sufficient condition for the composition $\sB'\by{\omega'_\sB} \Vec_R\by{-\otimes_R L}\Vec_L$ to be faithful; this condition is also necessary when $L=K$. (This will be used in a revised version of \cite{ff}.) 

We reap the fruits of our labour in Part II, showing in \S \ref{s5} that many motivic theories from \cite{adjoints} are stacks, and extending this to $1$-motives in \S \ref{s7} and to Nori motives in \S \ref{s8}. As indicated above, \S \ref{s6} shows that the construction of \cite[\S 4]{chow-lefschetz} is a ``stackification''. Finally, the appendix contains more foundational material, of a more abstract nature than the one in Part I.

\subsection*{Acknowledgements} I wish to thank Joseph Ayoub, Kevin Coulembier and Annette Huber for helpful exchanges.

\subsection*{Terminology} A $\otimes$-category is an additive symmetric monoidal category (the tensor structure being biadditive). A $\otimes$-functor $F$ between $\otimes$-categories $\sA,\sB$ is an additive symmetric monoidal functor; it is \emph{strong} if the structural morphisms $F(A)\otimes F(B)\to F(A\otimes B)$ are isomorphisms, \emph{lax} in general. If no adjective is used, it means by default that $F$ is strong; we use lax to emphasise the contrary. (See \cite[Ch. XI]{mcl}.)

\part{General theory}

\section{Stacks over a profinite group}\label{s1}

\subsection{The set up}\label{s1.1} Let $\sE$ be a category; recall from \cite[VI, \S\S 7, 8]{sga1} that there is a dictionary between fibered categories over $\sE$ and pseudo-functors $\sE\to \Cat$ whose comparison $2$-cocycle consists of natural isomorphisms; we shall adopt here the latter viewpoint, which is also the one of \cite{adjoints}.

Take $\sE=B\Pi$, where $\Pi$ is a profinite group and $B\Pi$ is its classifying topos, i.e. the category of finite continuous $\Pi$-sets. We want to give conditions for a (contravariant) pseudo-functor $\usA$ from $B\Pi$ to the $2$-category of categories to be a stack for the natural topology on $B\Pi$; this will be done in Theorem \ref{p5.3} when $\sA$ takes values in pseudo-abelian $\Q$-linear categories. Recall what being a stack means (\cite[Déf. II.1.2.1]{giraud},  \cite[Def. 026F]{stacks}); given $S\in B\Pi$:

\begin{enumerate}
\item For any $A,B\in \usA(S)$, the presheaf $(U\by{f} S)\mapsto \usA(U)(f^*A,f^*B)$ on $B\Pi/S$ is a sheaf;
\item Any descent datum relative to a cover $f:T\to S$ is effective.
\end{enumerate}

Here is the special case where $f$ is a Galois covering of connected (= transitive) $\Pi$-sets, with $\Gamma=\Gal(f)$; setting $\sA=\usA(S)$ and $\sA'=\usA(T)$:

\begin{enumerate}[label={\rm (\arabic*G)}] 
\item the map $\sA(A,B)\by{f^*} \sA'(f^*A,f^*B)$ induces an isomorphism $\alpha:\sA(A,B)\iso \sA'(f^*A,f^*B)^\Gamma$ for any $A,B\in \sA$;
\item any descent datum relative to $f$ is effective.
\end{enumerate}

In (1G), let us explain the action of $\Gamma$ on the right hand side: for each $g\in \Gamma$, the equality $fg=f$ and the pseudo-functor structure of $\usA$ yield a natural isomorphism $i_g:g^*f^*\overset{\sim}{\Rightarrow} f^*$; these are compatible with the natural isomorphisms $c_{g,h}:h^*g^*\overset{\sim}{\Rightarrow} (gh)^*$. To $\phi:f^*A\to f^*B$, one associates $\phi^g=i_g(B)(g^*\phi) i_g(A)^{-1}$ (right action!). If $\phi$ is of the form $f^*\psi$, then $\phi^g=\phi$ by the naturality of $i_g$.

The meaning of (2G) is the following: let $C\in\sA'$, provided with isomorphisms $b_g:g^*C\iso C$ verifying the usual $1$-coboundary condition
\[b_h\circ h^*b_g=b_{gh} \circ c_{g,h}\]
with respect to the $2$-cocycle $c_{g,h}$ (descent datum)\footnote{In \cite[\S 5]{chow-lefschetz}, a different convention is used.}. Then there exists $B\in \sA$ and an isomorphism $f^*B\iso C$ which induces an isomorphism of descent data, for the canonical descent datum on $f^*B$ implicitly used in the previous paragraph. Moreover, $B$ is unique up to unique isomorphism.

To formalise this, we introduce the category $\sA'[\Gamma]$ of descent data: an object is a descent datum as above, and morphisms are the obvious ones.\footnote{Note that $\sA'[\Gamma]$ is none else than the Grothendieck construction \cite[VI, \S 8]{sga1} on the pseudo-functor $\underline{\Gamma}\to \sA'$ giving the $g^*$, where $\underline{\Gamma}$ is the category with one object representing $\Gamma$.} If $(C,(b_g)),(D,(b'_g))\in \sA'[\Gamma]$ and $\phi\in \sA'(C, D)$, one defines $\phi^g$ for $g\in \Gamma$ as in the case of effective descent data, generalising the previous construction.

In fact (1G) and (2G) are sufficient to encompass (1) and (2), as shown by the following lemma. Let $B^\gal\Pi$ be the full subcategory of $B\Pi$ consisting of those (left) $\Pi$-sets $\Gamma$ where $\Gamma$ is a finite quotient of $\Pi$. We provide it with the topology induced by that of $B\Pi$ (any morphism is a cover). Note that every morphism in $B^\gal\Pi$ is a Galois covering.  For any site $\mathbf{S}$, let $\St(\mathbf{S})$ be the $2$-category of stacks over $\mathbf{S}$. Then

\begin{lemma}\label{l1.01}   The restriction $2$-functor $\St(B\Pi)\to \St(B^\gal\Pi)$ is an equivalence of $2$-categories.
\end{lemma}

\begin{proof}[sketch] This is a special case of the general fact that stacks over a site only depend on the topos associated to the site \cite[Th. II.3.5.1]{giraud}. For the reader's convenience, let us  describe a $2$-quasi-inverse:

Start from a stack $\usA$ on $B^\gal\Pi$. For $S\in B\Pi$ connected, let $P$ be the stabiliser of a point of $S$: this is an open subgroup of $\Pi$. Let $N\subset \Pi$ be an open normal subgroup contained in $P$: then $\Gamma=\Pi/N$ is finite, and $P/N$ (pseudo-)acts on $\usA(\Gamma)$ by restriction of the obvious action of $\Gamma$. Define $\usA(S)$ to be the category of descent data $\usA(\Gamma)(P/N)$. The stack property shows that it does not depend on the choice of $N$, up to canonical equivalence; choosing another base point in $S$ yields a conjugate of $P$ and also an equivalent category; this equivalence is unique up to a canonical isomorphism because the action of $P/N$ on $\usA(S)$ is canonically isomorphic to the identity. In general, we set $\usA(S)=\prod_i \usA(S_i)$ where the $S_i$'s are the connected components of $S$. We leave it to the reader to extend this construction to morphisms in order to define a pseudo-functor, and to check the stack property.
\end{proof}

Lemma \ref{l1.01} reduces the study of stacks over $B\Pi$ to that of stacks over $B^\gal\Pi$. Moreover, in much of the paper we shall only consider the functor $f^*$ for a fixed Galois $f:T\to S$, so it is convenient to abstract things a little more: thus our setting will be
\begin{itemize}
\item two categories $\sA$ and $\sA'$;
\item a pseudo-action of a finite group $\Gamma$ on $\sA'$; we say that $\sA'$ is a \emph{$\Gamma$-category};
\item a functor $f^*:\sA\to \sA'$ which pseudo-commutes with the action of $\Gamma$ (for its trivial action on $\sA$).
\end{itemize}

As above we have the category of descent data $\sA'[\Gamma]$, and there is a functor $\hat f^*:\sA\to \sA'[\Gamma]$ sending $A$ to $(f^*A,(i_g(A)))$; Condition (1G) amounts to say that $\hat f^*$ is fully faithful and (2G) amounts to say that it is essentially surjective. 
When this happens, we say that $f^*$ \emph{has descent}.

\subsection{Introducing an adjoint}\label{s1.0} As a motivation, we start with

\begin{lemma}\label{l1.0} The forgetful functor $\tilde f^*:\sA'[\Gamma]\to \sA'$ sending $(C,(b_g))$ to $C$ is faithful and conservative. If $\sA'$ has finite products, $\tilde f^*$ has the right adjoint $\tilde f_*:D\mapsto \prod_{g\in \Gamma} g^*D$ provided with the descent datum $(b_h)$ given by  the isomorphisms $c_{g,h}(D):h^*g^*D\iso (gh)^*D$ of \S \ref{s1.1}; the unit of this adjunction is given by the inverses of the $b_g$'s, and its counit by the projection on the factor $g=1$.
\end{lemma}

\begin{proof}
This is readily checked.
\end{proof}

Thus, in the presence of finite products in $\sA'$, the existence of a right adjoint to $f^*$ is necessary for descent to hold. We shall only assume $\sA'$ to have finite products from \S \ref{s1.7} onwards; for now, we just suppose that $f^*$ has a right adjoint $f_*$ with unit $\eta$ and counit $\epsilon$, and draw some corresponding identities.  

For a descent datum $(C,(b_g))$ and $g\in \Gamma$, we get an endomorphism $[g]$ of $f_*C$ corresponding to $\epsilon_C^g$ by adjunction; in formula:
\[[g] = f_*\epsilon_C^g\circ \eta_{f_*C}.\]

\begin{lemma}\label{l5.3} a) We have
\[\epsilon_C^g=\epsilon_C\circ f^*[g].\]
b) Let $A\in \sA$, $(C,(b_g))$ be a descent datum, and let $g\in \Gamma$. If $\phi:f^*A\to C$ and $\psi:A\to f_*C$ correspond to each other by adjunction, then $\phi^g$ and $[g]\circ \psi$ correspond to each other by adjunction. In particular (taking $A=f_*C$, $\psi=1_A$), we have $[gh]=[g][h]$ (sic) and $[g]$ is an automorphism. \\
c) Suppose that $C$ is an effective descent datum $f^*A$. For any $g\in \Gamma$, we have $\epsilon_{f^*A}^g\circ f^*\eta_A=1_{f^*A}$ and $[g]\circ \eta_A=\eta_A$.
\end{lemma}

\begin{proof} a) This is just the other adjunction identity relating $[g]$ and $\epsilon_C^g$. 

b) We have
\[\phi=\epsilon_{C} \circ f^*\psi,\quad \psi=f_*\phi \circ \eta_A.\]

The first identity yields
\[\phi^g=\epsilon_{C}^g \circ (f^*\psi)^g=\epsilon_{C}^g \circ f^*\psi.\]

By the second identity, the morphism corresponding to $\phi^g$ is then
\[f_*\phi^g\circ \eta_A=f_*\epsilon_{C}^g \circ f_*f^*\psi\circ  \eta_A=f_*\epsilon_{C}^g \circ \eta_{f_*C} \circ \psi=[g]\circ \psi\]
where we used the naturality of $\eta$. Hence also the last claim.

c) Indeed, for the first identity,
\[\epsilon_{f^*A}^g\circ f^*\eta_A=\epsilon_{f^*A}^g\circ (f^*\eta_A)^g=(\epsilon_{f^*A}\circ f^*\eta_A)^g=1_{f^*A}^g=1_{f^*A}\]
while the second one follows from b) applied to $\phi=1_{f^*A}$.
\end{proof}

\subsection{Cartesianity}\label{s1.7} Assume now that $\sA'$ has finite products. Let $C\in \sA'$ and $g\in \Gamma$. We define a morphism $f^*f_*C \to g^*C$ as the composition 
\[f^*f_*C\by{i_g(f_*C)} g^*f^*f_*C\by{g^*\epsilon_C}g^*C.\]

Collecting over $g$, we get a morphism
\begin{equation}\label{eq5.8}
f^*f_*C\to \prod_{g\in \Gamma} g^*C.
\end{equation}

\begin{defn}\label{d5.3} The functor $f^*$ is \emph{Cartesian} if ($f_*$ exists and) \eqref{eq5.8} is a natural isomorphism. \end{defn}

(Suppose that we are in the fibred situation described at the beginning of \S \ref{s1.1}. In view of the isomorphism of $\Pi$-sets $\coprod_{g\in \Gamma} T\iso T\times_S T$ given by $(g,y)\mapsto (y,gy)$, Definition \ref{d5.3} amounts to saying that the ``base change morphism'' $f^*f_*\Rightarrow (f\times_S 1)_* (1\times_S f)^*$ in the diagram
\[\begin{CD}
\sA'@>(1\times_S f)^*>> \sA(T\times_S T)\\
@V f_*VV @V(f\times_S 1)_* VV\\
\sA(S)@>f^*>> \sA(T)
\end{CD}\]
is an isomorphism. One should not confuse Definition \ref{d5.3} with the notion of a Cartesian morphism in a fibred category.)

Assume $f^*$ Cartesian, and let $(C,(b_g))$ be a descent datum. Composing with the $b_g$ in \eqref{eq5.8}, we get an isomorphism
\begin{equation}\label{eq5.8a}
f^*f_*C\by{u_C} \prod_{g\in \Gamma} C
\end{equation}
whose $g$-component is given, by definition, by $\epsilon_C^g$.

\begin{lemma}\label{l5.4} Let $h\in \Gamma$. Then  the action of $f^*[h]$ on the left hand side of \eqref{eq5.8a} amounts to the action of $h$ by right translation on the indexing set $\Gamma$ of its right hand side.
\end{lemma}

\begin{proof} Let $g\in \Gamma$. Using Lemma \ref{l5.3} a) and b), we find
\[\epsilon_C^g\circ f^*[h]=\epsilon_C\circ f^*[g]\circ f^*[h]=\epsilon_C\circ f^*[gh]=\epsilon_C^{gh}.\]
\end{proof}

\begin{rk}\label{r1.1} Assume that $\sA$ also has finite products and that $f^*$ commutes with products. In \eqref{eq5.8a}, take $C=f^*A$ for some $A\in \sA$. Then $\epsilon_{f^*A}=f^*\pi_A$ for 
\[\pi_A:\prod_{g\in \Gamma} A\to A\]
the projection on the factor $g=1$.
\end{rk}

Here are important consequences of cartesianity. First, a definition:

\begin{defn}\label{d1.1} a) (cf. {\cite[Def. 2.3.7]{bvk2}}). A functor $F:\sC\to \sD$ is \emph{dense} if any object of $\sD$ is isomorphic to a retract of $F(C)$ for some $C\in \sC$.\\
b) A category  is \emph{pointed} if it has an object which is initial and final.
\end{defn}

Suppose that $\sA'$ is pointed. Then any Hom set of $\sA'$ has a null element $0$; in particular, for $C\in \sC$ the projection $\prod_{g\in \Gamma} g^*C\to C$ on the $g=1$ factor has a section which is the identity on this factor and $0$ elsewhere. If $f^*$ is Cartesian, composing with the isomorphism \eqref{eq5.8} we get a section 
\begin{equation}\label{eq1.4}
\sigma_C:C\to f^*f_*C
\end{equation}
of $\epsilon_C$, which is natural in $C$.

\begin{lemma}\label{l1.1} If $\sA'$ is pointed and $f^*$ is Cartesian,\\
a) $f^*$ is dense;\\
b) for any category $\sB$ and two functors $a,b:\sA'\rightrightarrows \sB$, the natural map
\[f_!:\Hom(a,b)\to \Hom(af^*,bf^*)\]
between sets of natural transformations has a canonical retraction $\rho$; in particular, $f_!$ is injective. Moreover, $v\in  \Hom(af^*,bf^*)$ is in the image of $f_!$ if and only if it commutes with the $\epsilon_{f^*A}$ for all $A\in \sA$.
\end{lemma}

\begin{proof} a) Let $C\in \sA'$. Then $C$ is a retract of $f^*f_* C$ via the pair $(\epsilon_C,\sigma_C)$, where $\sigma_C$ is the section of \eqref{eq1.4}.

b) For $v\in \Hom(af^*,bf^*)$, define
\[\rho(v)_C =b(\epsilon_C)v_{f_*C} a(\sigma_C):a(C)\to b(C).\]

If $\psi:C\to D$ is a morphism, we have
\begin{multline*}
b(\psi)\rho(v)_C = b(\psi)b(\epsilon_C)v_{f_*C} a(\sigma_C)
=b(\epsilon_D)b(f^*f_*\psi)v_{f_*C} a(\sigma_C)\\
=b(\epsilon_D)v_{f_*D}a(f^*f_*\psi) a(\sigma_C)=b(\epsilon_D)v_{f_*D}a(\sigma_D)a(\psi) =\rho(v)_D a(\psi) 
\end{multline*}
so that $\rho(v)$ is a natural transformation. If $v=f_!(u)$ for some $u\in \Hom(a,b)$, then
\[\rho(v)_C =b(\epsilon_C)u_{f^*f_*C} a(\sigma_C)=u_Ca(\epsilon_C)a(\sigma_C)=u_C,\]
thus $\rho$ is indeed a retraction of $f_!$.

For the image of $f_!$, the condition is obviously necessary. Suppose that it holds. Then we have, for $A\in \sA$, 
\[f_!\rho(v)_A=\rho(v)_{f^*A} = b(\epsilon_{f^*A})v_{f_*{f^*A}} a(\sigma_{f^*A})= v_Aa(\epsilon_{f^*A}) a(\sigma_{f^*A})= v_A,\]
so $v=f_!\rho(v)$.
\end{proof}

\begin{rks} a) In Lemma \ref{l1.1} b), we could have used $\sigma$ instead of $\epsilon$ for the condition to be in the image of $f_!$.\\
b) $\rho$ does \emph{not} define a retraction of the functor $f_!:\Funct(\sA',\sB)\to \Funct(\sA,\sB)$ in general (it need not respect composition). However, it is compatible with composition with a further functor $\sB\to \sC$.
\end{rks}

\subsection{Traces}\label{s1.4} From now on, we assume $\sA$, $\sA'$ and $f^*$ (hence also $f_*$) additive. We write $\bigoplus$ instead of $\prod$. To formulate the result, we need a further definition:

\begin{defn}\label{d5.2} Suppose $f^*$ Cartesian. A \emph{trace structure} on $f^*$ is a natural transformation $\tr:f_*f^*\Rightarrow \Id_{\sA}$ such that, for any $A\in \sA$:
\begin{enumerate}
\item the composition
\begin{equation}\label{eq5.7}
A\by{\eta_A} f_*f^*A\by{\tr_A} A
\end{equation}
is multiplication by $|\Gamma|$;
\item the isomorphism \eqref{eq5.8a} for $C=f^*A$ converts $f^*\tr_A$ into the sum map.
\end{enumerate}
\end{defn} 

\subsection{Main result} 

\begin{prop}\label{p5.3a} 
a) If $f^*$ has descent, then
it is Cartesian and has a trace structure. \\
b) The converse is true if  $\sA$ is pseudo-abelian\footnote{Also sometimes called Karoubian or idempotent complete.} and $\Z[1/|\Gamma|]$-linear.
\end{prop}

\begin{proof}
a) Recall Lemma \ref{l1.0}.  Cartesianity is tautologically true, and the trace morphism is given by $\bigoplus_{g\in \Gamma} g^*C\by{(b_g)} C$ for $(C,(b_g))\in \sA'[\Gamma]$. Condition (1) of Definition \ref{d5.2} is immediate and Condition (2) is also tautological.

b)  We check Conditions (1G) and (2G) of \S \ref{s1.1}:

(1G)  By adjunction, the map $\sA(A,B)\to \sA'(f^*A,f^*B)$ may be rewritten as the map
\[\sA(A,B)\by{a} \sA(A,f_*f^*B)
\]
induced by the unit morphism $\eta_B$. Using \eqref{eq5.7}, we get a map $b$ in the opposite direction such that $ba$ is multiplication by $|\Gamma|$; hence $a$ is injective by hypothesis (for this it would suffice that $\sA(A,B)$ has no $|\Gamma|$-torsion).

I now claim that $ab=\sum_{g\in \Gamma} g$ for the action of $\Gamma$ on $\sA'(f^*A,f^*B)$ explained in \S \ref{s1.1}. By Lemma \ref{l5.3} b), it suffices to prove that the composition
\[f_*f^*B\by{\tr_B} B\by{\eta_B} f_*f^*B
\]
is $\sum_{g\in \Gamma} [g]$. By the faithfulness of $f^*$ which has just been established, it suffices to do this after applying $f^*$. By Condition (2) of the trace structure, this translates as a composition
\[\bigoplus_{g\in \Gamma} f^* B\by{\Sigma}  f^*B \by{\Delta}  \bigoplus_{g\in \Gamma} f^* B\]
in which $\Sigma$ is the sum map and $\Delta$ is the diagonal map by Lemma \ref{l5.3} c); the claim now follows from Lemma \ref{l5.4}.

Coming back to the proof of (1G), we find that the composition
\[\sA'(f^*A,f^*B)^\Gamma\inj \sA'(f^*A,f^*B)\by{ab} \sA'(f^*A,f^*B)^\Gamma\]
is also multiplication by $|\Gamma|$, hence the desired bijectivity of $\alpha$ in (1G).

(2G) Let $(C,(b_g))$ be a descent datum. Consider the idempotent $e_\Gamma=\frac{1}{|\Gamma|}\sum_{g\in \Gamma} [g]$ in $\End f_* C$, and let $A=\IM e_\Gamma$. The adjoint of the inclusion $\iota:A\inj f_*C$ yields a morphism $\tilde \iota:f^*A\to C$. Let us check that this is a morphism of descent data, and an isomorphism. 

The first point amounts to say that $\tilde\iota^g=\tilde\iota$ for all $g$ which, by Lemma \ref{l5.3} b), amounts to $[g]\circ \iota=\iota$ for all $g$: this is true by definition of $\iota$.

For the second point, we define a morphism $j:C\to f^*A$ as follows. Let $\pi:f_*C\to A$ be the projection associated to the idempotent $e_\Gamma$. Then $j$ is the composition
\[C\by{\Delta} \bigoplus_{g\in \Gamma} C\iso f^*f_*C\by{f^*\pi} f^*A\]
where the first morphism is the diagonal map and the second one is the inverse of the isomorphism \eqref{eq5.8a}. It remains to show that $j$ is inverse to $\tilde \iota$.

By the first point, we have $f^*[g]\circ f^*\iota=f^*\iota$ which means, by Lemma \ref{l5.4}, that all the components of $f^*\iota$ on  \eqref{eq5.8a} are equal, i.e. that $\Delta \epsilon_C f^*\iota = f^*\iota$. Therefore, with an abuse of notation,
\[j\tilde \iota = f^*\pi \Delta \epsilon_C f^*\iota= f^*\pi f^*\iota = 1_{f^* A}.\]

Finally, we have $ f^*\iota f^*\pi=f^*e = \frac{1}{|\Gamma|}\sum_{g\in \Gamma} f^*[g]$, hence
\[\tilde \iota j = \epsilon_C f^*\iota f^*\pi \Delta= \frac{1}{|\Gamma|}\sum_{g\in \Gamma}\epsilon_C f^*[g]\Delta=1_C\]
as desired.
\end{proof}

\begin{thm}\label{p5.3} 
a) If $\usA$ is a stack over $B\Pi$, then
\begin{thlist}
\item $\usA$ commutes with coproducts;
\item for any Galois covering $f:T\to S$ in $B\Pi$, with $S,T$ connected, $f^*$ is Cartesian and has a trace structure. 
\end{thlist}
b) Suppose that  $\usA(S)$ is pseudo-abelian for all $S$. Then the converse is true if  $\sA$ is $\Z[1/|\Gal(f)|]$-linear for any $f$ as in (ii) with $S=*$ (the one-point $\Pi$-set), e.g. if $\sA$ is $\Q$-linear.
\end{thm}

\begin{proof}  (i) has already been seen. The rest follows from Proposition \ref{p5.3a} and Lemma \ref{l1.01}.
\end{proof}

\subsection{A left adjoint structure on $f_*$}\label{left} The following is worth noting, but will not be used in the sequel. 

Suppose that $f^*$ is Cartesian and has a trace structure. For any $A\in \sA$, define $\epsilon'_A=\tr_A:f_*f^*A\to A$; for any $C\in \sA'$, define $\eta'_C:C\to f^*f_* C$ as the inclusion of the $g=1$ summand in the right hand side of \eqref{eq5.8}.

\begin{prop}\label{p2.4} The natural transformations $\epsilon'$ and $\eta'$ verify the (left) adjunction identities, provided $\sA(A,B)$ has no $|\Gamma|$-torsion for any $A,B\in \sA$. 
\end{prop}

\begin{proof} Let $A,C\in \sA\times \sA'$. We must show that the compositions
\begin{gather}
f^*A \by{\eta'_{f^* A}}f^*f_*f^*A\by{f^*\epsilon'_A} f^*A\label{la1}
\intertext{and}
f_*C\by{f_*\eta'_C} f_*f^*f_* C\by{\epsilon'_{f_*C}} f_*C\label{la2}
\end{gather}
are equal to the identity. For \eqref{la1}, this follows from Property (2) of Definition \ref{d5.2}. By Proposition \ref{p5.3a} b) and its proof, $f^*$ is faithful, hence it suffices to prove \eqref{la2} after applying this functor; using \eqref{eq5.8a}, this reduces to the previous case.
\end{proof}

\subsection{Algebras on a monad}\label{s1.3} Here we study the special case where $\sA'$ is the category $\sA^M$ of algebras over a additive monad $M$ in $\sA$ 
\[\sA^M=\{(A,\phi)\mid A\in \sA, MA\by{\phi} A) \}\]
where $\phi$ verifies certain identities \cite[VI.2, definition]{mcl}. The functor $f_*$ sends $(A,\phi)$ to $A$, while $f^*$ sends $A$ to the free algebra with underlying object $MA$, and the counit $\epsilon_A$ is given by the commutative square expressing the associativity of $\phi$ (in particular, $f_*\epsilon_A=\phi$). The following lemma is trivial:

\begin{lemma}\label{l1.2} The forgetful functor $f_*$ is faithful and conservative.\qed
\end{lemma}

A homomorphism $[\, ]:\Gamma\to \End(M)$ (see Definition \ref{dA.1}) yields a (strict) $\Gamma$-action on $\sA'$ ($g^*(A,\phi)=(A,\phi\circ [g])$, and then we are in a special case of the situation above; in particular, the category $\sA'[\Gamma]$ of descent data is defined:
\[\sA'[\Gamma]=\sA^M[\Gamma] =\{(A,\phi,b_g)\mid (A,\phi)\in \sA^M, b_g:A\to A \}\]
where $(b_g)$ verifies the identities $b_g \circ \phi\circ [g] = \phi\circ M(b_g)$ plus the cocycle condition; since the action of $\Gamma$ is strict, $g\mapsto b_g^{-1}$ is a group homomorphism.

\subsection{Codescent}\label{s1.2} The adjunction $(f^*,f_*)$ gives rise to a factorisation of $f_*$ into
\begin{equation}\label{eq1.1}
\sA'\by{K} \sA^M\by{U} \sA 
\end{equation}
for $M=f_*f^*$: the ``comparison'' functor $K$ is given by $K(C)=(f_*C,f_* \epsilon_C)$  (\emph{loc. cit.}, VI.3, Th. 1), and $U$ maps $(A,\phi)$ to $A$. It is $\Gamma$-equivariant.

\begin{prop}\label{p1.2} If $f^*$ is Cartesian and $\sA'$ is pseudo-abelian, then $K$ is an isomorphism of categories.
\end{prop}

\enlargethispage*{20pt}

Curiously, this proposition will only be used for a going up result in Theorem \ref{t4.1}.

\begin{proof} 
 Let $(\partial_0,\partial_1):C\rightrightarrows D$ be a pair of morphisms in $\sA'$. Assume that $(f_*\partial_0,f_*\partial_1)$ has a universal coequaliser in the sense of \cite[VI.6]{mcl}.  Let $\widehat{\sA'}$ be the additive dual of $\sA'$, $y:\sA'\to\widehat{\sA'}$ the additive Yoneda embedding,  and let $E$ be the coequaliser of $(y(\partial_0),y(\partial_1))$. Then $\bigoplus_{g\in \Gamma} g^*E$ is the coequaliser of $(\bigoplus_{g\in \Gamma}g^*y(\partial_0),\bigoplus_{g\in \Gamma}g^*y(\partial_1))$, where $g^*$ is the extension of $g^*:\sA'\to \sA'$ to $\widehat{\sA'}$ via $y$. By the cartesianity of $f^*$ and by the hypothesis on $(f_*\partial_0,f_*\partial_1)$ applied to $f^*$ and $y$, $\bigoplus_{g\in \Gamma} g^*E$ is representable, hence so is its direct summand $E$ since $\sA'$ is assumed to be pseudo-abelian. The conclusion now follows from Beck's theorem \cite[VI.7, Th. 1]{mcl}.
\end{proof}

\subsection{A stability property} Suppose that we have naturally commutative diagrams of additive categories and functors
\begin{equation}\label{eq2.5}
\begin{CD}
\sA'@>\gamma'>> \sB'\\
@Af^*AA @Af_\sB^*AA\\
\sA@>\gamma>> \sB
\end{CD}
\end{equation}
\begin{equation}\label{eq2.5a}
\begin{CD}
\sA'@>\gamma'>> \sB'\\
@Vf_*VV @V f^\sB_*VV\\
\sA@>\gamma>> \sB
\end{CD}
\end{equation}
where $f^\sB_*$ is right adjoint to $f_\sB^*$. The following proposition is trivial but very useful.

\begin{prop}\label{p1.3} a) Assume that $\gamma'$ is conservative. If $f_\sB^*$ is Cartesian, so is $f^*$.\\
b) (see also Theorem \ref{t4.1}). If moreover $\gamma$ is fully faithful, a trace structure on $f^*_\sB$ induces a unique trace structure on $f^*$.\qed
\end{prop}

\section{The monoidal case}\label{s.mon}

In this section, we assume that the additive categories $\sA$ and $\sA'$ are $\otimes$-categories and that the base change functor $f^*$ is a $\otimes$-functor: see \emph{Terminology} at the end of the introduction. We write $\un$ for the unit object of both $\sA$ and $\sA'$; this will not cause confusion (note that $f^*\un =\un$).

\subsection{Weak properties}\label{s2.1}
We have a ``projection morphism''
\[A\otimes f_*C\to f_*(f^*A\otimes C)\]
for $(A,C)\in \sA\times \sA'$, constructed as the adjoint of
\[f^*(A\otimes f_*C)\iso f^*A\otimes f^*f_*C \by{1\otimes\epsilon_C} f^*A\otimes C\]
where the first isomorphism is the inverse of the monoidal structure of $f^*$. For $C=\un$, we thus get a morphism
\begin{equation}\label{eq2.1}
w_A:A\otimes f_*\un\to f_*f^*A.
\end{equation}

By definition of $w_A$, we have

\begin{lemma}\label{l2.2} Modulo the monoidal structure of $f^*$, one has the identity $\epsilon_{f^*A} \circ f^*w_A=\epsilon_{f^*\un}\otimes 1_{f^*A}$.\qed
\end{lemma}

\begin{defn} We say that $f^*$ verifies the \emph{weak projection formula} if $w_A$ is an isomorphism for any $A\in \sA$, and is \emph{weakly Cartesian} if \eqref{eq5.8a} is an isomorphism for $C=\un$.
\end{defn}

\begin{lemma}\label{l2.1} Suppose that $f^*$ verifies the weak projection formula and is weakly Cartesian. Then $f^*$ is Cartesian if and only if, moreover, it is dense (Definition \ref{d1.1} a)).
\end{lemma}

\begin{proof} ``Only if'' follows from Lemma \ref{l1.1} a). If: for $A\in \sA$, consider the diagram
\begin{equation}\label{eq2.2}
\begin{CD}
f^*A\otimes f^*f_*\un@>f^*w_A>> f^*f_*f^*A\\
@V1\otimes u_{\un}VV @Vu_{f^*A}VV\\
f^*A\otimes \bigoplus\limits_{g\in \Gamma} \un@>\bigoplus_{g\in \Gamma} e_{f^*A} >> \bigoplus\limits_{g\in \Gamma} f^*A
\end{CD}
\end{equation}
where $e_{f^*A}$ is the unit constraint: it commutes by Lemma \ref{l2.2}. The bottom horizontal map is an isomorphism; so are the top and left vertical ones by assumption. Therefore $u_{f^*A}$ is also an isomorphism. By the denseness hypothesis, $u_C$ is then an isomorphism for every $C\in \sA'$.
\end{proof}

Suppose that $f^*$ is Cartesian and admits a trace structure in the sense of Definition \ref{d5.2}. Then there is a morphism $\tr:f_*\un\to \un$ such that
\begin{enumerate}[label={\rm (\arabic*u)}] 
\item the composition
\begin{equation*}
\un\by{\eta_{\un}} f_*\un\by{\tr} \un
\end{equation*}
is multiplication by $|\Gamma|$;
\item the isomorphism \eqref{eq5.8a} (for $C=\un$) converts $f^*\tr$ into the sum map.
\end{enumerate}

\begin{defn}\label{d2.1} We call this a \emph{weak trace structure}.
\end{defn}

Conversely:

\begin{prop}\label{p2.1} Suppose that $f^*$ verifies the weak projection formula and is weakly Cartesian. Then a weak trace structure yields a trace structure on $f^*$ by the formula $\tr_A=(1_A\otimes \tr)\circ w_A^{-1}$ for $A\in \sA$.
\end{prop}

\begin{proof} The first identity of Definition \ref{d5.3} is clear from (1u), and the second one follows from (2u) by using Diagram \eqref{eq2.2} again.
\end{proof}

\begin{ex}\label{ex2.1} $\sA=\Rep_K(G)$, $\sA'=\Rep_K(H)$ for two affine group schemes $G\supseteq H$ over a field $K$ such that $H\triangleleft G$ and $G/H\simeq \Gamma$. Then $f^*$ identifies with restriction $\Res_H^G$, whose right adjoint  is induction $\Ind_H^G$. Cartesianity and weak trace structure follow respectively from the Mackey formula and Frobenius reciprocity.
\end{ex}

\begin{cor}\label{c2.1} If $f^*$ has descent, it verifies the weak projection formula, is  
Cartesian and has a weak trace structure; the converse is true if $\sA$ is pseudo-abelian and $\Z[1/|\Gamma|]$-linear.
\end{cor}

\begin{proof} Collect Proposition \ref{p5.3a} 
and Proposition \ref{p2.1}.
\end{proof}

\subsection{A monoidal retraction}\label{s2.2} We come back to the situation  of Lemma \ref{l1.1} b), where we now assume that $\sB$ is a $\otimes$-category and that $a,b$ are $\otimes$-functors; hence so are also $af^*$ and $bf^*$. We write $\Hom^\otimes(af^*,bf^*)$ and $\Hom^\otimes(a,b)$ for the sets of \emph{not necessarily unital} $\otimes$-natural transformations, so that $f_!$ carries the latter to the former.

\begin{prop} \label{p2.2} The retraction $\rho$ of Lemma \ref{l1.1} b) carries $\Hom^\otimes(af^*,bf^*)$ to $\Hom^\otimes(a,b)$.  If moreover $f^*$ verifies the weak projection formula, then $v\in  \Hom(af^*,bf^*)$ is in the image of $f_!$ if and only if it commutes with $\epsilon_{f^*\un}$.
\end{prop}

\begin{proof} Let $u\in \Hom^\otimes(af^*,bf^*)$, and let $C,D\in \sA'$. We have to show that
\[
\rho(u)_{C\otimes D} =\rho(u)_C\otimes \rho(u)_D.
\]

Using the $\otimes$-structures of $f^*$, $a$ and $b$, this amounts to the equality
\[
b(\epsilon_C\otimes \epsilon_D)\circ u_{f_*C\otimes f_*D} \circ a(\sigma_C\otimes \sigma_D)=
b(\epsilon_{C\otimes D})\circ u_{f_*(C\otimes D)}\circ a(\sigma_{C\otimes D}).
\]

For $C,D\in \sA'$, the morphism
\[f^*(f_*C\otimes f_*D)\iso f^*f_*C\otimes f^*f_*D\by{\epsilon_C\otimes \epsilon_D} C\otimes D,\]
where the first map is the inverse of the (strong) monoidal structure on $f^*$, yields by adjunction a morphism
\begin{equation}\label{eq6.1}
f_*C\otimes f_*D\to f_*(C\otimes D)
\end{equation}
(lax monoidal structure on $f_*$). This yields a lax monoidal structure $\mu$ rendering the diagram
\[\xymatrix{
f^*f_*C\otimes f^*f_* D\ar[dr]_{\epsilon_C\otimes \epsilon_D}\ar[rr]^\mu&& f^*f_*(C\otimes D)\ar[dl]^{\epsilon_{C\otimes D}}\\
&C\otimes D
}\]
commutative. Using the isomorphisms \eqref{eq5.8}, this translates to the following diagram:
\[\xymatrix{
\bigoplus\limits_{g,h}g^*C\otimes h^*D\ar[dr]_{\epsilon_C\otimes \epsilon_D}\ar[rr]^\mu&& \bigoplus\limits_g g^*C\otimes g^*D\ar[dl]^{\epsilon_{C\otimes D}}\\
&C\otimes D
}\]
where $\mu$ identifies to the obvious projection. We have a dual commutative diagram
\[\xymatrix{
\bigoplus\limits_{g,h}g^*C\otimes h^*D&& \bigoplus\limits_g g^*C\otimes g^*D\ar[ll]_\lambda\\
&C\otimes D\ar[ul]^{\sigma_C\otimes \sigma_D}\ar[ur]_{\sigma_{C\otimes D}}
}\]
where $\lambda$ is the obvious inclusion. Therefore, it suffices to prove the identity
\[
b(\mu)\circ u_{f_*C\otimes f_*D} \circ a(\lambda)=
 u_{f_*(C\otimes D)}
\]
which follows from the naturality of $u$ and the identity $\mu\lambda=1$.

The last point follows from Lemma \ref{l1.1} b), the weak projection formula and Lemma \ref{l2.2}.
\end{proof}

\begin{rk} If $u$ is unital, $\rho(u)$ is not necessarily unital. 
In the situation of Example \ref{ex2.1}, $f_*\un=K[\Gamma]$ with its natural left action by $G$ and $u$ identifies with an element of $G(K)$, while $\epsilon$ sends $\sum_{g\in \Gamma} \lambda_g [g]$ to $\lambda_1$ and $\sigma(1) =[1]$. Thus $\rho(u)_\un$ is $1$ if $u\in H(K)$ and $0$ otherwise. It follows that $\rho(u)=0$ in the latter case.
\end{rk}

 \subsection{Monoidal codescent}\label{s2.3} This is the pendant of Subsection \ref{s1.2}. Here we simply remark that, if $f^*$ verifies the weak projection formula., the monad $M$ of § \ref{s1.2} is 
\[MA = f_*\un \otimes A\]
and $\sA^M$ is the category of modules in $\sA$ over the monoid $f_*\un$ \cite[VII.4]{mcl}, see Lemma \ref{lA.2}. This monoid is commutative because the monoidal structures are symmetric.

\subsection{Artin objects}

\begin{defn}\label{l5.1a} An \emph{Artin object for $f^*$} is an object $A\in \sA$ such that $f^*A\simeq n\un$ for some $n\ge 0$. Artin objects form a (full) rigid $\otimes$-subcategory of $\sA$, denoted by $\sA^0(f^*)$.
\end{defn}

Let $A\in \sA^0(f^*)$ be an Artin object. Then $\Gamma$ acts on $\Hom(\un,f^*A)\simeq \Hom(f^*\un,f^*A)$ as in §\ref{s1.1}. This defines a $Z$-linear $\otimes$-functor 
\begin{align}
\sA^0(f^*)&\to \Rep_Z(\Gamma)\label{eq5.1}\\
A&\mapsto \sA'(\un,f^*A)\notag
\end{align}
where $Z=\End_{\sA(S)}(\un)=\End_{\sA^0(f^*)}(\un)$, and the right hand side is the category of representations of $\Gamma$ on free finitely generated $Z$-modules.

\begin{lemma}\label{l5.2}  Under the hypotheses of Corollary \ref{c2.1}, this functor is an equivalence of categories.
\end{lemma}

\begin{proof} Let $\sA'[\Gamma]^0$ be the full subcategory of $\sA'[\Gamma]$ consisting of those objects $(C,(b_g))$ such that $C$ is isomorphic to $n\un$ for some $n\ge 0$. Then \eqref{eq5.1} factors as a composition
\[\sA^0(f^*)\by{\hat{f}^*} \sA'[\Gamma]^0\by{V} \Rep_Z(\Gamma)\]
with $V(C)=\sA'(\un,C)$ as before. By definition of $Z$, $V$ is an equivalence of categories and so is $\hat{f}^*$ by Corollary \ref{c2.1}.
\end{proof}

\subsection{An exactness result} We go back to the situation of \S \ref{s2.2}. For $u\in \Hom^\otimes(af^*,bf^*)$, we write $u_{|\sA^0(f^*)}=1$ if $u_A:af^*A\to bf^*A$ is the identity for any $A\in \sA^0(f^*)$ modulo the isomorphisms
\[af^*A\simeq a(n\un) \simeq n\un_\sB, \quad  bf^*A\simeq b(n\un) \simeq n\un_\sB.\]

\begin{thm}\label{t2.1} Let $\Hom^{\otimes,\un}(a,b)$ be the subset of $\Hom^\otimes(a,b)$ formed of unital $\otimes$-natural transformations. Suppose that $f^*$ verifies the weak projection formula. Then $u\in \break\Hom^\otimes(af^*,bf^*)$ is of the form $f_!v$ for a (unique) $v\in \Hom^{\otimes,\un}(a,b)$ if and only $u_{|\sA^0(f^*)}=1$.
\end{thm}

\begin{proof} Uniqueness follows from the existence of the retraction $\rho$ of Lemma \ref{l1.1} b). 
The condition is obviously necessary, and its sufficiency follows from  Proposition \ref{p2.2} plus the hypothesis on $u$, since $f_*\un\in \sA^0(f^*)$ by the isomorphism  \eqref{eq5.8}.  
\end{proof}

\section{Morphisms of stacks}\label{s3}

\subsection{A trivial lemma} The following is obvious:

\begin{lemma}\label{l4.1} Let $F:\usA\to \usB$ be a morphism of stacks over a site. If $F$ is faithful (resp. fully faithful, an equivalence of categories) locally, it is so globally.\qed
\end{lemma}

\subsection{Universal extension}\label{s3.1} 
Let $\sA,\sA',f^*$ be as in Section \ref{s.mon}. Let $\sB$ be a $\otimes$-category and $\gamma:\sA\to \sB$ be a $\otimes$-functor. We are going to do a reverse construction to that of \S \ref{s1.2}. 

Recall from \S \ref{s2.3} that \eqref{eq6.1} provides $f_*\un$ with a commutative monoid structure. Then $R=\gamma(f_*\un )$ is a commutative monoid of $\sB$, and $\gamma$ induces a functor
\[\sA^{f_*\un}\to \sB^R=:\sB'\]
hence a functor
\begin{equation}\label{eq3.1}
\gamma':\sA'\to \sB'
\end{equation}
obtained by composing with the comparison functor $K$ of \eqref{eq1.1}: explicitly,
\begin{equation}\label{eq3.3}
\gamma' C=(\gamma f_*C,\gamma f_*\epsilon_C).
\end{equation}

It comes with a naturally commutative diagram \eqref{eq2.5a} in which $f^\sB_*$ is the forgetful functor.

Recall that $f^\sB_*$ has the left adjoint $f_\sB^*:X\mapsto (R\otimes X,\mu\otimes 1_X)$ where $\mu$ is the multiplication of $R$: this is a special case of \cite[VI.2, Th. 1]{mcl}. Therefore we get a base change morphism 
\begin{equation}\label{eq2.3}
f_\sB^*\gamma\Rightarrow \gamma' f^*
\end{equation}
fitting in Diagram \eqref{eq2.5} (so far it is not necessarily invertible).

Suppose that $\sB$ has coequalisers (\emph{e.g} that it is abelian). Since $f_*\un$ is commutative, so is $R$; by Proposition \ref{pA.1}, $\sB'$ acquires a $\otimes$-structure with unit $R$, and $f_\sB^*$ is a $\otimes$-functor.

The action of $\Gamma$ on $f_*\un$ (\S \ref{s1.0}) carries over to $R$ via $\gamma$ and defines a pseudo-action of $\Gamma$ on $\sB'$ such that $\gamma'$ is $\Gamma$-equivariant. In particular, the category of descent data $\sB'[\Gamma]$ is defined (see \S \ref{s1.3}).

\begin{prop}\label{p4.2} Assume that $f^*$  
verifies the weak projection formula. Then\\
a) The natural transformation \eqref{eq2.3} is invertible.\\
b) If moreover $f^*$ is Cartesian and $\sB$ has cokernels,  $\gamma'$ is a $\otimes$-functor.\\
c) If moreover $\sA$ is $\Z[1/|\Gamma|]$-linear and $\gamma'$ is dense, $f_\sB^*$ has descent.
\end{prop}

\begin{proof} a) After composition with $f^\sB_*$, the value of \eqref{eq2.3} on $A\in \sA$ becomes
\[R\otimes \gamma(A)\to \gamma(f_*f^* A)\]
which, via the weak projection formula, is the strong monoidality isomorphism of $\gamma$; since $f^\sB_*$ is conservative (Lemma \ref{l1.2}), we are done.

 b) We first provide $\gamma'$ with an (a priori lax) symmetric monoidal structure. Let $C,D\in \sA'$. The lax monoidal structure \eqref{eq6.1} yields a $0$-sequence
 \[f_*C \otimes f_*\un\otimes f_*D\to f_*C\otimes f_* D\to f_*(C\otimes D)\]
 where the first map is the difference of the $f_*\un$ actions on $f_*C$ and $f_* D$. Applying $\gamma$ and using its strong monoidality, 
 we get another $0$-sequence
 \[\gamma f_*  C \otimes R\otimes \gamma f_*  D\to \gamma f_*  C\otimes \gamma f_*   D\to \gamma f_*  (C\otimes D)\]
  which induces the desired natural transformation (compare \eqref{eq3.3} and \eqref{eq7.2}):
 \begin{equation}\label{eq2.4}
 \gamma' C\otimes \gamma' D\to \gamma'(C\otimes D).
 \end{equation}
 
 By a) and the strong monoidality of $f_\sB^*$, $\gamma'\circ f^*$ is strongly monoidal: in other terms, \eqref{eq2.4} is an isomorphism when $C$ and $D$ are of the form $f^*A$ and $f^*B$, hence in general by Lemma \ref{l1.1} a).

c) If $\sA$ is $\Z[1/|\Gamma|]$-linear, so is $\sB$; it is also pseudo-abelian since it has cokernels. By Corollary \ref{c2.1}, it suffices to see that $f_\sB^*$ verifies the weak projection formula, has a weak trace structure and  is Cartesian. The first fact is a tautology, the second follows from the same property for $f_*\un$, as does the  weak cartesianity of $f_\sB^*$. But since $\gamma'$ and $f^*$ are dense, so is their composition and thus so is $f_\sB^*$ as well; hence $f_\sB^*$ is Cartesian by Lemma \ref{l2.1}.
 \end{proof}

\begin{rk} The density hypothesis on $\gamma'$ in c) seems artificial, even though it is easy to verify in practice. I don't know how to avoid it.
\end{rk}

We now have a going-down and going-up theorem:

\begin{thm}\label{t4.1} Under all the hypotheses of Proposition \ref{p4.2}, (i.e. assuming that $\sA$ is $\Z[1/|\Gamma|]$-linear, that $f^*$ verifies the weak projection formula and is Cartesian, and that $\gamma'$ is dense),  $\gamma$ is fully faithful if and only if $\gamma'$ is fully faithful.
\end{thm}

\begin{proof} ``If'' follows from Lemma \ref{l4.1} and Proposition \ref{p4.2}. For ``only if'', the full faithfulness of $\gamma$ implies that of $\sA^{f_*\un}\to \sB^R$. The conclusion then follows from Proposition \ref{p1.2}.
\end{proof}

\subsection{``Universal'' property of the universal extension}

\begin{prop}\label{p2.3} Consider Diagram \eqref{eq2.5}. Let $\sC$ be a $\otimes$-category and let $a:\sA'\to \sC$, $b:\sB\to \sC$ be two $\otimes$-functors, provided with a natural $\otimes$-transformation $v:b\gamma\Rightarrow af^*$. Suppose that, as in Theorem \ref{t4.1}, all the hypotheses of Proposition \ref{p4.2} are verified and that, moreover, $\sC$ has cokernels. Then there exists a unique $\otimes$-functor $b':\sB'\to \sC$ such that $b=b'f_\sB^*$; it comes with a canonical natural $\otimes$-transformation $u:b' \gamma'\Rightarrow a$.
\end{prop}

\begin{proof} Applying $a$ to the counit of the adjunction $(f^*,f_*)$ yields a morphism
\[a^*f^*f_*\un\to a^*\un=\un.\]

Composing it with $v$ gives another morphism
\[b(R)=b\gamma f_*\un\to \un\]
which is a homomorphism of monoids by construction. By Corollary \ref{pA.2}, this yields the first claim. By Proposition \ref{p4.2} a), we then get a natural $\otimes$-transformation $b'\gamma' f^*\Rightarrow af^*$, and Proposition \ref{p2.2} provides $u$.
\end{proof}

\section{Tannakian categories}\label{s4}

\subsection{The set-up} Let $\sA,\sA',f^*$ be again as in Section \ref{s.mon}. We add some assumptions: $\sA,\sA'$ are abelian and rigid, and $Z(\sA)\iso Z(\sA')=K$, where $K$ is a field of characteristic $0$. Throughout, we suppose that $f^*$ satisfies the hypotheses of Theorem \ref{p5.3} b), hence satisfies descent.

\subsection{Going up}\label{s4.3}

Let $\omega:\sA\to \Vec_L$ be a fibre functor, where $L$ is an extension of $K$ (thus $\sA$ is a Tannakian category over $K$).  Write $E=\omega(f_*\un)$.

\begin{lemma}\label{l4.4}  $E$ is an étale $L$-algebra of dimension $|\Gamma|$.
\end{lemma}

\begin{proof} By Cartesianity and the projection formula, we have
\[f_*\un\otimes f_*\un=f_*f^*f_*\un=f_*\prod_\Gamma \un=\prod_\Gamma f_*\un\]
hence
\[E\otimes_L E\iso \prod_\Gamma E\]
where the homomorphism is given by $r\otimes s\mapsto (rg(s))_{g\in \Gamma}$. Here the action of $\Gamma$ on $E$ is induced by its action on $f_*\un$. The claims follow.
\end{proof}

\begin{rk} Thus $E$ is a Galois $\Gamma$-algebra over $L$ in the sense of \cite[1.3]{bs}.
\end{rk}

As in \eqref{eq3.1}, we get  a $\otimes$-functor
\[\tilde\omega':\sA'\to (\Vec_L)^E=\Vec_E\]
where the right hand side denotes the $\otimes$-category of $E$-modules which are finite-dimensional over $L$ (\emph{i.e.} of finite type over $E$).

\begin{lemma}\label{l4.2} The functor $\tilde\omega'$ is exact and faithful.
\end{lemma}

\begin{proof} Let $U:\Vec_E\to \Vec_L$ be the forgetful functor. We have $U\tilde \omega' = \omega f_*$. The right hand side is exact  and faithful as a composition of two such functors. But $U$ is also faithful and exact, hence faithfully exact, hence the conclusion.
\end{proof}

\subsection{The neutral case} Here we assume $L=K$. Let  $G=\Aut^{\otimes}(\omega)$ be the Tannakian group of $\omega$ and $H=\Aut^{\otimes}(\omega')$ that of $\omega'$  (recall that every $\otimes$-endomorphism of $\omega$ or  $\omega'$ is an automorphism, hence unital, by rigidity \cite[Rk. 2.18]{dm}). By Tannakian duality, we may then write $\sA=\Rep_K(G)$ and $\sA'=\Rep_K(H)$.

The $\otimes$-functor $f^*$ induces a homomorphism $i:H\to G$.  The equivalence of Lemma \ref{l5.2} is induced by $\omega'$ since $\omega=\omega'\circ f^*$ (indeed, $\sA'(\un,B)$ is functorially isomorphic to $\omega'(B)$ for any split $B\in \sA'$). Whence a homomorphism $p:G\to \Gamma$.

\begin{thm}\label{t5.1a} 
The sequence $1\to H\by{i} G\by{p} \Gamma\to 1$ is exact.
\end{thm}

\begin{proof}  By Theorem \ref{t2.1}, it suffices to show that $p$ is epi.  By \cite[Prop. 2.21 (a)]{dm}, we must show that every subobject $B\in \sA$ of an object $A\in \sA^0(f^*)$ belongs to $\sA^0(f^*)$;  but this is obvious since $\un$ is simple in  $\sA'$ by  \cite[Prop. 1.17]{dm}. 
\end{proof}

\subsection{Globalisation}\label{s4.2} Let $K$ be a field of characteristic $0$, and let $\usA$ be a pseudo-functor from $B^\gal\Pi$ to the $2$-category $\Ex^\rig(K)$ of rigid abelian $\otimes$-categories $\sC$ with $\End_\sC(\un)=K$, $\otimes$-functors and $\otimes$-natural transformations. Let $\sA=\usA(*)$, where $*$ is the terminal object. Define $\sA_\infty$ as $2\text{-}\colim_{T\in B^\gal(\Pi)} \usA(T)$: it belongs to $\Ex^\rig(K)$.

Let $\omega_\infty:\sA_\infty\to \Vec_K$ be a fibre functor to the category of finite-dimensional $K$-vector spaces: by restriction, it defines a fibre functor $\omega_T$ on $\usA(T)$ for every $T$. For $T=*$, we write  $\omega_T=\omega$.  Let  $G=\Aut^\otimes(\omega)$ be the Tannakian group of $\omega$ and $H=\Aut^{\otimes}(\omega_\infty)$ the one of $\omega_\infty$. 

\begin{lemma}\label{l3.1} The natural morphism $H\to \lim_T \Aut^{\otimes}(\omega_T)$ is an isomorphism.
\end{lemma}

\begin{proof} It suffices to verify this on $R$-points for any $K$-algebra $R$. Then it follows from the definition of $\sA_\infty$.
\end{proof}

\begin{thm}\label{t5.1} Suppose that suppose that $f^*$ satisfies the hypotheses of Theorem \ref{p5.3} b) for any Galois $f:T\to *$ .  Then the sequence
\[1\to H\by{i} G\by{p} \Pi\to 1\]
is exact.
\end{thm}

\begin{proof} In view of Lemma \ref{l3.1}, this follows from Theorem \ref{t5.1a}.
\end{proof}

\begin{rk} Even if it is not obvious, this proof is inspired by Ayoub's proof of the corresponding theorem in \cite[Prop. 5.7]{ayoub}, using Hopf algebras. He explained me a version using ind-Tannakian categories, which inspired the retraction of Lemma \ref{l1.1} b). Here is this argument, translated from French: 

One has a 
morphism of ind-Tannakian categories $e^*:T \to T'$, a fibre  functor $w'^*:T' \to \text{Vect}_K$ [the category of all small $K$-vector spaces] and one sets $w^*=w'^*\circ e^*$. Assume that for every object $M\in T'$, the morphism
$e^*e_*M\otimes_{e^*e_*K}K\to M$ is an isomorphism. Since $w^*w_*K=w'^*e^*e_*w'_*K$ 
we get that $w^*w_*K\otimes_{w^*e^*e_*K}K\simeq w'^*w'_*K$ as desired. 
\end{rk}

\begin{rk} There is an obvious extension of Theorem \ref{t5.1} to the case of $\otimes$-morphisms between two fibre functors, as in Theorem \ref{t2.1}. Formulating it is left to the reader. Similarly for another extension to Tannakian monoids for fibre functors on not necessarily rigid $\otimes$-categories.
\end{rk}

\subsection{Enrichments}\label{s4.4} Consider now a factorisation of $\omega$
\begin{equation}\label{eq3.2}
\sA\by{\gamma} \sB\by{\omega_\sB}\Vec_L
\end{equation}
where $\sB$ is another Tannakian category over $K$ and $\gamma,\omega_\sB$ are exact and faithful $\otimes$-functors.

Let $\sB'$ be the universal extension of \S \ref{s3.1}, and take the notation of  \eqref{eq2.5} and \eqref{eq2.5a}. Since $\sB$ is rigid, its tensor structure is exact.  By Lemmas  \ref{lA.1a} d) and \ref{lA.2}, $\sB'$ is abelian and the forgetful functor $f^\sB_*$ is exact. 

Suppose that $\omega$ is the restriction to $\sA$ of a fibre functor $\omega':\sA'\to \Vec_L$. By Proposition \ref{p2.3}, applied with $(\sC,a,b,\beta)\equiv (\Vec_L,\omega,\omega_\sB,1)$, there exists a unique $\otimes$-functor $\omega'_\sB:\sB'\to \Vec_L$ such that $\omega'=\omega'_\sB f_\sB^*$. It is provided with a natural transformation $v:\omega'_\sB\gamma'f^*\Rightarrow \omega'f^*$ such that $v_{\un}$ is the morphism $\omega(\epsilon_{\un}):E=\omega(f_*\un)\to \omega(\un)=L$.

\begin{lemma}\label{l4.3} The functor $\omega'_\sB$ is the composition of $\tilde \omega'$ and the functor $L\otimes_R -:\Vec_R\to \Vec_L$. It is exact.
\end{lemma}

\begin{proof} The first claim follows by the functoriality of the construction of Proposition \ref{p2.3}. In the composition, the first functor is exact by Lemma \ref{l4.2}, and the second is exact because the homomorphism $R\to L$ is flat, thanks to Lemma \ref{l4.4}.
\end{proof}

Contrary to Lemma \ref{l4.2}, $\omega'_\sB$ is not faithful in general, for example if $\sB=\Vec_L$!
The following proposition gives a case where it is. Recall that a rigid $\otimes$-category $\sC$ is \emph{connected} if $Z(\sC):=\End_\sC(\un)$ is a field.

\begin{prop}\label{p4.3} In the above situation, the functor $\omega'_\sB$ is faithful if and only if $\sB'$ is connected. A sufficient condition is that the restriction of $\gamma$ to $\sA^0(f^*)$ is full. 
\end{prop}

\begin{proof}  Since $\Vec_L$ is connected, the condition is necessary; the converse follows from \cite[Prop. 1.19]{dm}. If the restriction of $\gamma$ to Artin objects is full, then the map 
\[
K=\End_{\sA'}(\un)=\sA(\un,f_*f^*\un)\by{\gamma} \sB(\un,\gamma_s f_*f^*\un)
\simeq \sB(\un,f^\sB_*f_\sB^*\un))=\End_{\sB'}(\un)
\]
is bijective, where we used Proposition \ref{p4.2} a) for the isomorphism.
\end{proof}

We come back to the neutral case, write $\sB=\Rep_K(G')$ and let $\gamma^*:G'\to G$ be the homomorphism dual to $\gamma$. Let $H'=\Ker(G'\to G\to \Gamma)$.

\begin{prop} The functor $\omega'_\sB$ factors as a composition
\begin{equation}\label{eq4.1}
\sB'\by{\pi} \Rep_K(H')\by{\bar \omega'_\sB} \Vec_K 
\end{equation}
where $\pi$ is a Serre localisation and $\bar \omega'_\sB$ is faithful. Moreover, $\pi$ is an equivalence of categories if and only if $G'\to \Gamma$ is epi. In particular, the fullness condition is also necessary in Proposition \ref{p4.3}. 
\end{prop}

\begin{proof} \eqref{eq4.1} is the canonical factorisation of the exact functor $\omega'_\sB$ into a Serre localisation followed by a faithful functor. To identify the middle category with $\Rep_K(H')$, we apply Corollary \ref{pA.2} to the restriction functor $\Rep_K(G')\to \Rep_K(H')$ to factor it through $\sB'$. In the last statement, sufficiency follows from Proposition \ref{p4.3}. For necessity, suppose that $\pi$ is an equivalence. Then $Z=Z(\sB')$ is a field, and we have a factorisation of the identity
\[K=Z(\Rep_K(H))\by{\gamma'} Z\by{\omega'_\sB} Z(\Vec_K)=K\]
hence $Z=K$ and $\gamma'$ is surjective. As in the proof of Proposition \ref{p4.3}, this gives that $\sA(\un,f_*\un)\by{\gamma} \sB(\un,R)$ is bijective, from which the fullness of $\gamma_{\mid \sA^0(f^*)}$ easily follows; in turn, this is equivalent to the surjectivity of $G'\to \Pi$.
\end{proof}

\part{Applications}

\section{The general layout}\label{s0}

Let $k$ be a base field. The idea of the applications which follow is to start from the basic functoriality of schemes (or pairs of schemes) over a finite Galois extension $l/k$, and to transport it to categories of motives through the motive functor. This leads to the following caveat:

In the said categories of schemes, naïve restriction of scalars is left adjoint to restriction of scalars. If the motive functor is contravariant, it will convert this functor into a right adjoint, and we can directly apply the framework of \S\S \ref{s1} and \ref{s.mon}. This is the case for Chow-Lefschetz motives (\S \ref{s6}) and Nori motives (\S \ref{s8}), but not for the theories of \cite{adjoints} studied in \S \ref{s5}, where the choice was that of a covariant motive functor. This means that in the latter case one must replace these categories by their opposites; of course, this does not affect the stack property. Thus cartesianity will follow from cartesianity for $l$-schemes $X$:
\begin{equation}\label{eq5.2}
\coprod_{g\in \Gamma} g_* X\iso X_{(k)}\otimes_k l
\end{equation}
where $\Gamma=\Gal(l/k)$ and $g_*$ is the base change given by $g:\Spec l\to \Spec l$ for $g\in \Gamma$; \eqref{eq5.2} is itself induced by the special case $X=\Spec l$ (Galois theory). Similarly,   the weak projection formula will follow from the equality for $k$-schemes $Y$:
\begin{equation}\label{eq5.4}
(Y\otimes_k l)_{(k)}= Y\times_{\Spec k} \Spec l.
\end{equation}

Here we write $(-)\otimes_k l$ for extension of scalars from $k$ to $l$, and $(-)_{(k)}$ for the naïve restriction of scalars from $l$ to $k$ (i.e., composing with the morphism $\Spec l\to \Spec k$).

\section{Motivic theories}\label{s5} The following generalises Theorem \ref{t0} of the introduction:

\begin{thm}\label{t3.1} All motivic theories $\sA$ of \cite[Th. 4.3 a)]{adjoints} are stacks for the étale topology on $\Spec k$ provided they are $\Q$-linear. In particular, this is the case for pure motives à la Grothendieck for any adequate equivalence relation.
\end{thm}

\begin{proof} By Lemma \ref{l1.01} and Corollary \ref{c2.1}, it suffices to check that, for any finite Galois extension $f:T=\Spec l\to S=\Spec k$, $f^*$  verifies the weak projection formula, is Cartesian and has a weak trace structure. Use $M$ generically to denote the ``motive'' functor $\Sm(-)\to \sA(-)$. As explained in \S \ref{s0}, we replace $\sA(-)$ by $\sA^\op(-)$ to make $M$ contravariant. By \cite[Th. 4.1]{adjoints} and its proof, $f_*$ exists and 
commutes with naïve restriction of scalars on $\Sm(-)$ via $M$.

That \eqref{eq5.8} is a natural isomorphism is checked on pseudo-abelian generators of $\sA$. Also, $f^*$ commutes with Tate twists when they are present in the theory $\sA$. We thus may take $A=M(X)$ for $X\in \Sm(k)$ or $\Sm^\proj(k)$, and we are reduced to \eqref{eq5.2}. Similarly, \eqref{eq2.1} reduces to \eqref{eq5.4} by the monoidality of $M$. Finally, we define the weak trace $\tr$ by using the (finite) correspondence given by the transpose of graph of the projection $\Spec l\to \Spec k$. The axioms of a weak trace structure follow readily.
\end{proof}

\begin{rk} The same result holds for the motivic theories of Deligne \cite{dm} and André \cite{pour}, with the same proof.
\end{rk}

\section{Chow-Lefschetz motives}\label{s6}

\subsection{The associated stack}\label{s4.1} Let $\usA_0$ be a fibred category over a site $\Sigma$. Recall \cite[Th. II.2.1.3]{giraud} that there is an ``associated stack'' $\usA$ together with a fibred functor $\usA_0\to  \usA$ which is $2$-universal for fibred functors from $\sA_0$ to stacks. The stack $\usA$ is constructed from $\usA_0$ in two steps:

\begin{description}
\item[Associated prestack (cf. \protect{\cite[Lemma II.2.2.2]{giraud}})] $\usA_1$: same objects as $\usA_0$; for $S\in \Sigma$ and $X,Y\in \usA_0(S)$, $\usA_1(S)(X,Y)$ is the sheaf associated to the presheaf $(T\to S)\mapsto \usA_0(T)(X_T,Y_T)$.
\item[Associated stack (cf. \protect{\cite[Lemma 3.2]{lmb}})]  starting from $\sA_1$, for $S\in \Sigma$ an object of $\usA(S)$ is a descent datum of $\usA_1$ for a suitable  cover $(U_i)_{i\in I}\to S$; morphisms are given by refining covers. This operation is fully faithful (loc. cit., Remark 3.2.1).
\end{description}

In the case $\Sigma=B\Pi$, these two constructions translate as follows, with the notation of Section \ref{s1}: in Step 1, one replaces the groups $\usA_0(S)(A,B)$ by $\colim_T \usA_0(T)(f^*A,f^*B)^{\Gal(f)}$, where $f:T\to S$ runs through the (finite) Galois coverings of $S$; for Step 2, we take the $2$-colimit of the categories of descent data on $\sA_1$. One could do both constructions in one gulp, but this would not be convenient for the next subsection.

\subsection{The case of Chow-Lefschetz motives} In \cite{chow-lefschetz} we introduced categories of ``Chow-Lefschetz motives'' $\LMot_\sim(k)$ over a field $k$ (modulo an adequate equivalence relation $\sim$) in two steps: a) by defining ``crude'' categories $\LMot_\sim(k)_0$ \cite[\S 4.1]{chow-lefschetz}; b) by refining this construction \cite[\S 4.2]{chow-lefschetz}. 

\begin{prop}\label{p4.1}
$\LMot_\sim$ is the stack associated to $(\LMot_\sim)_0$.
\end{prop}

\begin{proof} Here we use implicitly Lemma \ref{l1.01} to consider only finite Galois extensions $l/k$. We first prove that $\LMot_\sim$ is a stack. This is essentially done in \cite{chow-lefschetz}: the descent property for morphisms is loc. cit., (4.4) and the effectivity of descent data is shown in the proof of  Theorem 5 in loc. cit., §5.5 in the same way as here (we were inspired here by this argument). Alternately we may apply Corollary \ref{c2.1} of the present paper just as in the proof of Theorem \ref{t3.1}, using the right adjoint of \cite[Lemma 4.5]{chow-lefschetz} (note that the isomorphism \eqref{eq5.8} is explicitly proven in this lemma).

In remains to show that the canonical fibred functor $(\LMot_\sim)_0\to \LMot_\sim$ induces an equivalence on the associated stacks; it suffices to do it for the fibred functor $(\LCorr_\sim)_0\to \LCorr_\sim$ on categories of correspondences. After forming the associated prestack $(\LCorr_\sim)_1$ as in §\ref{s4.1}, this functor becomes fully faithful. Let $l/k$ be finite Galois, with group $\Gamma$; the $\Gamma$-equivariant fully faithful functor $\LCorr_\sim(l)_1\to \LCorr_\sim(l)$ induces a fully faithful functor on the categories of descent data, hence $\LCorr_\sim(k)_1\to \LCorr_\sim(k)$ factors through a fully faithful functor $\LCorr_\sim(l)_1[\Gamma]\to \LCorr_\sim(k)$, and then through a fully faithful functor $2\text{-}\colim_l \LCorr_\sim(l)_1[\Gamma]\to \LCorr_\sim(k)$. 

For its essential surjectivity, let $A$ be an object of $\LCorr_\sim(k)$: by definition, it is an abelian scheme over an étale $k$-algebra $E$. Choose $l/k$ and $\Gamma$ as above such that $l$ splits $E$. Then the $l$-scheme $B=\coprod_{\sigma\in \text{Mor}_k(E,l)} \sigma^*A$ is provided with a canonical descent datum $(b_g)_g\in \Gamma$, given by the action of $\Gamma$ on $\text{Mor}_k(E,l)$, and the object $(B,(b_g))\in \LCorr_\sim(l)_1[\Gamma]$ maps to $A$.
\end{proof}

\section{$1$-motives}\label{s7}

Let $\Mot_1(k)$ be the category of Deligne $1$-motives over a field $k$. Here there is no need to use the present theory: 

\begin{thm} The assignment $l\mapsto \Mot_1(l)$, where $l$ runs through all finite Galois extensions of $k$, is a stack (compare Lemma \ref{l1.01}).
\end{thm}

\begin{proof} This is trivial: we may view $\Mot_1(k)$ as a full subcategory of the category of arrows of the category of locally quasi-projective group schemes. When $k$ varies, the latter is a Galois stack by \cite[VIII, Cor. 7.6]{sga1}, hence so is $\Mot_1$ as well.
\end{proof}

\section{Nori motives}\label{s8}

We refer to \cite[Ch. 9]{h-ms} for a construction of Nori's category of mixed motives  over a subfield $k$ of $\C$. We shall denote it here by $\MM(k)$ (it is denoted by $\sM\sM(k)$ in \emph{loc. cit.}).

Let $l/k$ be a finite extension, corresponding to $f:\Spec l\to \Spec k$. We write  $f^*:\MM(k)\to \MM(l)$ for the base change functor denoted by $\res_{l/k}$ in \cite[Lemma 9.5.1]{h-ms}. 

\begin{prop}\label{p6.1}  The functor $f^*$ has a right adjoint  $f_*$. If $l/k$ is Galois, $f^*$ satisfies the weak projection formula, is Cartesian and has a weak trace structure in the sense of Definition \ref{d2.1}.
\end{prop}

The proof will show that $f_*$ coincides with the functor $\cores_{l/k}$ of \cite[Prop. 9.5.3]{h-ms}.

\begin{proof} It is  variant of that of Theorem \ref{t3.1}.  The full subcategory $\sC$ of $\MM(l)$ formed of those $M$'s such that $f_*$ is defined at $M$ is closed under kernels; more precisely, if $0\to M'\to M\to M''$ is an exact  sequence in $\MM(l)$ such that $M,M''\in \sC$, then $f_*M$ is given by $\Ker(f_*M\to f_* M'')$. By a result of Fresán and Jossen \cite[Th. 6.3]{fj}, any object of $\MM(l)$ is a subobject of an object of the form $H^i_\Nori(X,Y)(n)$ for a triple $(X,Y,i)$, hence has a copresentation by objects of this form. Therefore it suffices to check that $f_*$ is defined at such objects.

Recall that $f^*H^i_\Nori(X,Y)(n)=H^i_\Nori(X_l,Y_l)(n)$, where $X_l=X\otimes_k l$ for a $k$-scheme $X$ \cite[Lemma 9.5.1]{h-ms}. For a $l$-triple $(X,Y,i)$, write $M=H^i_\Nori(X,Y)(n)$ and define $f_*M=H^i_\Nori(X_{(k)},Y_{(k)})(n)$ where $(-)_{(k)}$ denotes the (naïve) restriction of scalars. Define a counit morphism $\epsilon:f^*f_* M\to M$ by the canonical morphism of triples \[(X,Y,i)\to ((X_{(k)})_l,(Y_{(k)})_l,i).\] 

We must show that the composition
\begin{equation}\label{eq6.2}
\MM(k)(N,f_*M)\by{f^*} \MM(l)(f^*N,f^*f_*M)
\by{\epsilon_*}  \MM(l)(f^*N,M)
\end{equation}
is an isomorphism for any $N\in \MM(k)$. Since $f^*$ is exact \cite[Lemma 9.5.1]{h-ms}, we reduce to the case where $N$ is of the form $H^j_\Nori(X',Y')(m)$ by \cite[Th. 6.1]{fj} (dual to the previous theorem). Since twisting is invertible and commutes with $f^*$, playing with  powers of $\G_m$ or $\P^1$ (see \cite[9.3.7 and 9.3.8]{h-ms}) we may even assume $n=m=0$.

For $N$ as above, define a unit morphism $\eta:N\to f_*f^*N$ by the canonical morphism of triples $((X'_l)_{(k)},(Y'_l)_{(k)},j)\to (X',Y',j)$. We get another composition
\begin{equation}\label{eq6.3}
\MM(l)(f^*N,M)\by{f_*} \MM(k)(f_*f^*N,f_*M)\by{\eta^*}  \MM(k)(N,f_*M)
\end{equation}
and it suffices to show that it is inverse to \eqref{eq6.2}. Since $\epsilon$ and $\eta$ are the counit and unit of an adjunction between categories of triples, this is true by the functoriality of $H^i_\Nori(-)(n)$.

Checking that \eqref{eq5.8} and \eqref{eq2.1} are isomorphisms is done in exactly the same way: note that since $f^*$ is exact, $f_*$ and hence $f^*f_*,f_*f^*$ are left exact. Thus we may reduce to the case of objects of the form $H^i_\Nori(X_l,Y_l)(n)$ by diagram chase,  using \cite[Th. 6.3]{fj} again. The isomorphism \eqref{eq5.8} follows from the same in the categories of triples (Galois descent). For  \eqref{eq2.1}, same reasoning by using the partial monoidality of $H^*_\Nori$ \cite[Prop. 9.3.1]{h-ms}, for which we remark that $(\Spec l,\emptyset,i)$ is a good pair. 

This proves everything, except the existence of a weak trace structure. For this we need to define a morphism
\[\tr_{l/k}:H^0_\Nori(\Spec l,\emptyset)=f_*\un_{\MM(l)}\to \un_{\MM(k)}=H^0_\Nori(\Spec k,\emptyset)\]
with the properties (1u) and (2u) stated before Definition \ref{d2.1}. We do as in the proof of Proposition \ref{p4.1}. The two properties are proven in the same way (or deduced from the existence of a functor from Chow motives to $\MM$).
\end{proof}

Proposition \ref{p6.1} holds for Nori motives $\MM(-,A)$ with coefficients in any commutative ring $A$ (same proof). Along with Lemma \ref{l1.01} and Corollary \ref{c2.1}, this yields:

\begin{thm} \label{t6.1} If $A$ is a $\Q$-algebra, the assignment $l\mapsto \MM(l,A)$ defines a stack over $(\Spec k)_\et$.\qed
\end{thm}

\begin{rks} Theorem \ref{t5.1} provides a proof of \cite[Th. 9.1.16]{h-ms}.
\end{rks}

\appendix
\section{Monads and monoids}\label{s9} I put here things I didn't find in \cite{mcl}.

\begin{lemma}\label{lA.1} Let $(T,\eta,\mu)$ be a monad in a category $\sC$ \cite[VI.1]{mcl}. Then the sequence
\[T^3\rightrightarrows T^2\by{\mu} T\]
is a split coequaliser in the sense of \cite[VI.6]{mcl}, where the first pair of arrows is $(T\mu,\mu T)$. More precisely, applying this sequence to any object of $\sC$ yields a split coequaliser.
\end{lemma}

\begin{proof} Define $s:T\to T^2$ and $t:T^2\to T^3$ by $s=\eta T$ and $t=\eta T^2$, and check the identities.
\end{proof}

\begin{lemma}\label{lA.1a} With the notation of Lemma \ref{lA.1}, let $\sC^T$ be the category of $T$-algebras \cite[VI.2]{mcl}. Then\\
a) The forgetful functor $U:\sC^T\to \sC$ is faithful, conservative and reflects equalisers. In particular, if $\sC$ has equalisers then so has $\sC^T$.\\
b) If $T$ preserves coequalisers, then $U$ reflects coequalisers; hence $\sC^T$ has coequalisers if $\sC$ does.\\
c) If $\sC$ and $T$ are additive,  $\sC^T$ is additive.\\
d) If $\sC$ is abelian and $T$ is right exact, then $\sC^T$ is abelian and $U$ is exact.
\end{lemma}

\begin{proof} a) The first two properties are obvious. For the third, let $(a;b):(C_1,\phi_1)\rightrightarrows (C_2,\phi_2)$ be two parallel arrows in $\sC^T$, and suppose that $(Ua,Ub)$ has an equaliser $c:C\to C_1$. The composition
\[TC\by{Tc} TC_1\by{\phi_1} C_1\]
is equalised by $Ua$ and $Ub$, hence factors uniquely through $C$; one checks that the resulting morphism $\phi:TC\to C$ defines a $T$-algebra, and then that this $T$-algebra is an equaliser. 

b) For $(a,b)$ as in a), suppose that $(Ua,Ub)$ has a coequaliser $d:C_2\to D$. By hypothesis, $Td$ is a coequaliser of $(TUa,TUb)$, hence $d \phi_2$ factors uniquely through a $\psi:TD\to D$. One sees that this is a $T$-algebra by observing that $T^2$ also respects coequalisers, and then that it is a coequaliser.

c) is easy and left to the reader. For d), the characterisation of an abelian category by the isomorphism of coimages onto images yields that $\sC^T$ is abelian via a), b) and c). Since $U$ is a right adjoint, it is left exact, and it remains to show that it preserves epimorphisms, which follows from b) by viewing $0$ as the cokernel of an epimorphism.
\end{proof}

\begin{rk} If the conclusions of d) hold, then conversely $T$ (assumed to be additive) is right exact as the composition of $U$ and its (right exact) left adjoint.
\end{rk}

\begin{defn}\label{dA.1} Let $(T',\eta',\mu')$ be another monad in $\sC$. \\
a) Let $u,v:T\Rightarrow T'$ be two natural transformations. We write
\[u\bullet v:T^2\to {T'}^2\]
for either of the compositions $T^2\by{Tv}TT'\by{uT'} {T'}^2$, $T^2\by{uT}T'T\by{T'v} {T'}^2$.\\
b) A \emph{morphism} from $T$ to $T'$ is a natural transformation $u:T\Rightarrow T'$ such that
\begin{thlist}
\item $u\eta=\eta'$;
\item $u\mu=\mu'(u\bullet u)$.
\end{thlist}
\end{defn}

\begin{prop}\label{pA.3} Let $u:T\to T'$ be a morphism of monads as in Definition \ref{dA.1} b). Then there is a canonical functor $u^*:\sC^{T'}\to \sC^T$ of ``restriction of scalars''. If $\sC$ has coequalisers and $T'$ is right exact,, $u^*$ has a left adjoint $u_!$ (``extension of scalars'').
\end{prop}

\begin{proof} If $(C,\psi)$ is a $T'$-algebra, then $(C,\psi\circ u_C)$ is a $T$-algebra. This defines $u^*$. Suppose now that $\sC$ has coequalisers. For $(C,\phi)\in \sC^{T'}$, let $D$ is the coequaliser of $(T'(\phi),\mu'_C\circ T'(u_C)):T'TC\rightrightarrows T'C$. If $\pi:T'C\to D$ is the associated morphsm, $\pi\circ \mu'_C$ equalises $({T'}^2(\phi),T'\mu'_C\circ {T'}^2(u_C))$ because the diagram
\[\begin{CD}
{T'}^2TC&
\begin{smallmatrix}
{T'}^2(\phi)\\
\rightrightarrows\\
T'(\mu'_C)\circ {T'}^2(u_C)
\end{smallmatrix}
& {T'}^2C@>T'(\pi)>> T'D\\
@V\mu'_{TC}VV @V\mu'_CVV\\
T'TC&
\begin{smallmatrix}
T'(\phi)\\
\rightrightarrows\\
\mu'_C\circ T'(u_C)
\end{smallmatrix}
& T'C@>\pi>> D
\end{CD}\]
commutes thanks to the associativity axiom for $\mu'$ and its naturality. The right exactness assumption on $T'$ implies that the top row is a coequaliser, hence the diagram can be completed by a unique map $\psi:T'D\to D$, that one checks to be a $T'$-algebra morphism. The composition
\[C\by{\eta'_C} T'C\by{\pi} D\]
defines a morphism of $T$-algebras $(C,\phi)\to u^*(D,\psi)$ and one checks that it is universal.
\end{proof}

\begin{lemma}\label{lA.2} Let $\sB$ be a monoidal category, and let $(R,\eta,\mu)$ be a monoid in $\sB$ \cite[VII.3]{mcl}. Then $TX=R\otimes X$ defines a monad in $\sB$ provided with a natural isomorphism $TX\otimes Y\iso T(X\otimes Y)$. Conversely, any monad provided with such a natural isomorphism is of this form.
\end{lemma}

\begin{proof} The first claim is easy to check by comparing the axioms of a monad and a monoid. For the converse, let $R=T\un$. Then, for any $X\in \sB$, one has $TX\iso T(\un\otimes X) \iso R\otimes X$.
\end{proof}

\begin{prop}\label{pA.1} In the situation of Lemma \ref{lA.2}, suppose $\sB$ symmetric and $R$ commutative (\emph{i.e.}, $\mu \circ \sigma = \mu$ where $\sigma$ is the switch of $R\otimes R$).\\
a) (cf. \cite[Prop. 4.1.10]{bran}). Let $\sB^R={}_R\Lact$ be the category of left actions by $R$ \cite[VII.4]{mcl}, and let $m_R:\sB\to \sB^R$, $X\mapsto (R\otimes X,\mu\otimes 1_X)$ be the left adjoint to the forgetful functor $U$ (ibid.). Suppose also that $\sB$ has coequalisers and that $-\otimes R$ is right exact (e.g., that $\otimes$ itself is right exact). Then there is a unique symmetric monoidal structure on $\sB^R$ such that $m_R$ is a strong $\otimes$-functor.\\
b) The functor $U$ reflects dualisability. In particular, if $\sB$ is rigid, so is $\sB^R$.\\
c) (cf. loc. cit., Rem. 4.1.11). Let $S$ be a second commutative monoid in $\sB$, and let $\phi:R\to S$ be a homomorphism of monoids. Then there is a unique $\otimes$-functor $\phi_*:\sB^R\to\sB^S$ such that $m_S=\phi_*\circ m_R$.
\end{prop}

\begin{proof} a) \emph{Existence}. By the hypotheses, to a left action $\nu$ of $R$ on $X\in \sB$ corresponds a right action given by
\[X\otimes R\by{\sigma} R\otimes X\by{\nu} X\]
where $\sigma$ is the symmetry of $\sB$, and conversely. We use this remark to switch sides without mention.

 For $(X,\nu_X),(Y,\nu_Y)\in \sB^R$, define
\begin{equation}\label{eq7.2}
X\otimes_R Y = \Coker(X\otimes R\otimes Y\rightrightarrows X\otimes Y)
\end{equation}
where $\Coker$ means coequaliser and the two maps are $\nu_X\otimes 1_Y,1_X\otimes \nu_Y$. Define $\nu:X\otimes R\otimes Y\to X\otimes_R Y$ via  either of these two maps; their associativity shows that $\nu$ factors through a morphism $\nu_{X\otimes_R Y}:R\otimes X\otimes_R Y\to X\otimes_R Y$. One checks easily that $(X\otimes_R Y,\nu_{X\otimes_R Y})\in  \sB^R$ and that the axioms of a symmetric monoidal structure, with unit $R$, are satisfied.

Let $X_0,Y_0\in \sB$. We must identify $R\otimes X_0\otimes Y_0$ with the coequaliser of
 \[(R\otimes X_0)\otimes R \otimes (R\otimes Y_0)\rightrightarrows (R\otimes X_0)\otimes (R\otimes Y_0)\]
 where, modulo the symmetric constraints, the two morphisms are respectively induced by $\mu\otimes 1_{X_0}$ and $\mu\otimes 1_{Y_0}$. By Lemmas \ref{lA.2} and \ref{lA.1}, this is true when $X_0=Y_0=\un$, and then it is even a split coequaliser. This  coequaliser remains a split coequaliser after tensoring it with $X_0\otimes Y_0$.
 
\emph{Uniqueness}. Let $\bullet$ be another solution. Let $(X,\nu_X),(Y,\nu_Y)$ be as above. By adjunction, we have a canonical morphism $X\otimes Y=U(X,\nu_X)\otimes U(Y,\nu_Y)\to U(X\bullet Y,\nu_{X\bullet Y})=:X\bullet Y$; since $R$ must be the unit of $\bullet$, this morphism must equalise the two morphisms of \eqref{eq7.2}, hence induce a morphism $\theta:(X\otimes_R Y,\nu_{X\otimes_R Y})\to (X,\nu_X)\bullet (Y,\nu_Y)$, which must be an isomorphism when $(X,\nu_X)$ and $(Y,\nu_Y)$ are in the image of $m_R$. In general, the counits $m_RU(X,\nu_X)\to (X,\nu_X)$ and $m_RU(Y,\nu_Y)\to (Y,\nu_Y)$ become split epis after applying $U$; therefore, $U(\theta)$ is an isomorphism and so is $\theta$ since $U$ is conservative.

b) Let $(X,\nu_X)\in \sB^R$, where $X$ has the dual $X^*$. Define $\nu_{X^*}$ as the composition
\begin{multline*}
R\otimes X^*\by{1\otimes \eta} R\otimes X^*\otimes X\otimes X^*\by{1\otimes \sigma\otimes 1} R\otimes X\otimes X^*\otimes X^*\\
\by{\nu_X\otimes 1} X\otimes X^*\otimes X^*\by{\sigma\otimes 1} X^*\otimes X\otimes X^*\by{\epsilon\otimes1} X^*
\end{multline*}
where $\eta,\epsilon$ are the unit and counit of the duality structure for $(X,X^*)$ and $\sigma$ is the symmetry. One verifies that this makes $(X^*,\nu_{X^*})$ dual to $(X,\nu_X)$.

c) By a), we may view $S$ as a commutative monoid in $\sC=\sB^R$ via $\phi$, and get a strong $\otimes$-structure on $m_S(\sC):\sC\to \sC^S$. It remains to observe that $\sC^S=\sB^S$.
\end{proof}

\enlargethispage*{30pt}

This construction has a universal property \cite[Prop. 5.3.1]{bran}\footnote{I thank Kevin Coulembier for this reference.}:

\begin{cor}\label{pA.2} In the situation of Proposition \ref{pA.1}, let  $\sC$ be another $\otimes$-category with coequalisers. Then any strong $\otimes$-functor from $F:\sB\to \sC$, provided with an algebra homomorphism $\beta:F(R)\to \un_\sC$, induces a unique strong $\otimes$-functor $\tilde F:\sB^R\to \sC$ provided with a natural $\otimes$-isomorphism $F\iso \tilde F\circ m_R$, and conversely.
\end{cor}

\begin{proof} Apply Proposition \ref{pA.1} c) to $(\sB,R,S)\equiv (\sC,F(R),\un_\sC)$. In the other direction, the counit of the adjunction $(m_R,U)$ yields a morphism 
\[m_R(R)=m_R U(\un_{\sB^R})\to \un_{\sB^R}\]
to which we apply $\tilde F$.
\end{proof}


\begin{thebibliography}{11}
\bibitem{pour} Y. André {\it Pour une théorie inconditionnelle des motifs}, Publ. Math. IHÉS {\bf 83} (1996), 5--49.
\bibitem{bs} E. Bayer, J.-P. Serre {\it Torsions quadratiques et bases normales autoduales}, Amer. J. Math. {\bf 116}, 1994, 1--64.
\bibitem{ayoub} J. Ayoub {\it Weil cohomology theories and their motivic Hopf algebroids}, preprint, 2023. \url{https://arxiv.org/abs/2312.11906}.
\bibitem{bvk2} L. Barbieri-Viale, B. Kahn {\it Universal Weil cohomology}, preprint, 2024, \url{https://arxiv.org/abs/2401.14127}.
\bibitem{bran} M. Brandenburg {\it Tensor categorical foundations of algebraic geometry}, preprint, 2014, \url{https://arxiv.org/abs/1410.1716}.
\bibitem{dm} P. Deligne and J. Milne {\it Tannakian categories}, {\it in} Hodge Cycles, Motives, and Shimura Varieties, Lect. Notes in Math. {\bf 900}, Springer, 1982, 101--228.
\bibitem{den-mur} C. Deninger and J. P. Murre {\it Motivic decomposition of abelian schemes and the Fourier transform},  J. Reine Angew. Math. {\bf 422} (1991), 201--219.
\bibitem{fj} J. Fresán, P. Jossen {\it Algebraic cogroups and Nori motives}, preprint, 2018, \url{https://arxiv.org/abs/1805.03906}.
\bibitem{giraud} J. Giraud Cohomologie non abélienne, Springer, 1965.
\bibitem{sga1} A. Grothendieck Revêtements étales et groupe fondamental (SGA 1), with two exposés by Mme M. Raynaud, revised edition, Doc. Math. {\bf 3}, SMF, 2003, 119--152.
\bibitem{h-ms} A. Huber, S. Müller-Stach Periods and Nori motives, Springer, 2017.
\bibitem{jannsen} U. Jannsen Mixed motives and algebraic $K$-theory, Lect. Notes in Math. {\bf 1400}, Springer.
\bibitem{adjoints} B. Kahn {\it Motifs et adjoints}, Rend. Sem. mat. Univ. Padova {\bf 139} (2018), 77--128.
\bibitem{chow-lefschetz} B. Kahn {\it Chow-Lefschetz motives}, Indag. Math. 2024 (Murre memorial volume).
\bibitem{ff} B. Kahn {\it The fullness conjectures for products of elliptic curves},  J. reine angew. Mathematik (Crelle) {\bf 819} (2025), 301--318.
\bibitem{lmb} G. Laumon, L. Moret-Bailly Champs algébriques, Springer, 2000.
\bibitem{mcl} S. Mac Lane Categories for the working mathematician, Grad. Texts in Math. {\bf 5}, Springer (2nd ed.), 1998.
\bibitem{stacks} The Stacks project, \url{https://stacks.math.columbia.edu}.
\end{thebibliography}
\end{document}